\documentclass[11pt,twoside,a4paper]{article}

\usepackage[utf8]{inputenc}
\usepackage{amsmath,amsfonts,amssymb,amsthm}
\usepackage{enumerate,lineno,setspace,float}
\usepackage[numbers,sort&compress]{natbib}
\usepackage[small,compact]{titlesec}
\usepackage[capitalize]{cleveref} 
\usepackage{graphicx,float,url}
\usepackage{geometry}
\usepackage{fancyhdr}
\usepackage{latexsym}
\usepackage{setspace}
\usepackage{tikz}
\usepackage{caption}

\usepackage[capitalize]{cleveref}

\newcommand{\Z}{\mathbb{Z}}

\newtheorem{theorem}{Theorem}[section]
\newtheorem{lemma}[theorem]{Lemma}

\newtheorem{definition}{Definition}

\newtheorem{observation}[theorem]{Observation}

\input amssym.def
\input amssym.tex

\usepackage{xcolor}
\usepackage[normalem]{ulem}

\newcommand{\be}{\begin{equation}}
\newcommand{\ee}{\end{equation}}
\newcommand{\ben}{\begin{equation*}}
\newcommand{\een}{\end{equation*}}
\newcommand{\bea}{\begin{eqnarray}}
\newcommand{\eea}{\end{eqnarray}}
\newcommand{\bean}{\begin{eqnarray*}}
\newcommand{\eean}{\end{eqnarray*}}
\def\non{\nonumber}

\def\qed{\hfill$\Box$\vspace{11pt}}

\def\diam{{\rm diam}}

\def\span{{\rm span}}
\def\rn{{\rm rn}}

\def\({\left(}
\def\){\right)}
\def\[{\left[}
\def\]{\right]}

\newcommand{\N}{\mathbb{N}}
\let\leq =\leqslant
\let\geq =\geqslant

\begin{document}

\title{Radio Number for the Cartesian Product of Two Trees}
\author{\textbf{Devsi Bantva} \\ Department of Mathematics \\ Lukhdhirji Engineering College, Morvi - 363642 \\ Gujarat, India \\ E-mail : devsi.bantva@gmail.com \\ \\ 

\textbf{Daphne Der-Fen Liu}\thanks{Research is partially supported by the National Science Foundation under grant DMS 1600773 and Cal State LA Properest Faculty Fellow grant.} \\ California State University, Los Angeles, USA \\
E-mail : dliu@calstatela.edu

}

\pagestyle{myheadings}
\markboth{\centerline{Devsi Bantva and Daphne Der-Fen Liu}}{\centerline{Radio Number for the Cartesian Product of Two Trees}}

\date{}
\openup 0.5\jot
\maketitle

\begin{abstract} 
Let $G$ be a simple connected graph. For any two vertices $u$ and $v$, let $d(u,v)$ denote the distance between $u$ and $v$ in $G$, and let $\diam(G)$  denote the diameter of $G$. A  \textit{radio-labeling} of $G$ is a function $f$ which assigns to each vertex a non-negative integer (label) such that for every distinct vertices $u$ and $v$ in $G$, it holds that  $|f(u)-f(v)| \geq  \diam(G) - d(u,v) +1$. The  \textit{span} of $f$ is the difference between the largest and smallest labels of $f(V)$.  The radio number of $G$, denoted by $\rn(G)$, is the smallest span of a radio labeling admitted by $G$. In this paper, we give a lower bound for  the radio number of the Cartesian product of two trees. Moreover, we present three  necessary and sufficient conditions, and three sufficient conditions for the product of two trees to  achieve this bound. Applying these results, we determine the radio number of the Cartesian product of two stars as well as a path and a star.  \\


\end{abstract}

\section{Introduction}
Radio labeling of graphs is motivated by the channel assignment introduced by Hale \cite{Hale}. In the channel assignment problem, the task is to assign a channel to each of the given set of stations or transmitters so that interference is avoided and the spectrum of the channels is minimized. The level of interference between two  stations are related to the proximity of  their locations. The closer the locations the stronger interference might occur. In order to avoid stronger inference, the separation of the channels assigned to the pair of stations has to be relatively larger. 

One can model the above problem by representing each station by a vertex and connecting two very close stations by an edge.  A {\it radio labeling} of a graph $G$ is a mapping, $f:V(G) \to \Z^+$, so that the following holds for any pair of distinct vertices:
\begin{equation}\label{def:rn}
|f(u)-f(v)| \geq \diam(G)+1-d(u,v),    
\end{equation}
where $d(u,v)$ is the distance between $u$ and $v$, and $\diam(G)$ is the diameter of $G$. The span of $f$, denoted by $\span(f)$, is defined as $\span(f) = \max\{|f(u)-f(v)| : u, v \in V(G)\}$. The {\it radio number} of $G$ is $$
\rn(G) = \min\{ \ {\rm span} (f) : f \mbox{ is a radio labeling of } G \}. 
$$

The notion of radio labeling was introduced by Chartrand et al. \cite{ Chartrand2}. Since then the radio number for special families of graphs has been studied widely in the literature (cf. \cite{Benson, Cada, Heuvel, Kral}). The radio number of cycles and paths were  determined by Liu and Zhu \cite{Daphne2}. Khennoufa and Togni \cite{Khennoufa} studied the radio number for hypercubes by using generalized binary Gray codes. Ortiz et al. \cite{JP} investigated the radio number of generalized prism graphs. Niedzialomski \cite{graceful} studied radio graceful graphs (where $G$ admits a surjective radio labeling) and showed that the Cartesian  product of $t$ copies of a complete graph is radio graceful for certain $t$, providing infinitely many examples of radio graceful graphs of arbitrary diameters. Zhou \cite{Zhou} investigated the radio number for Cayley graphs. Recently, Bantva and Liu discussed the radio number for block graphs and line graphs of trees in \cite{Bantva3}. For  positive integers $m, n \geqslant 3$, the {\it toroidal grid} $G_{m,n}$ is the Cartesian product of cycles $C_m$ and $C_n$. Morris et al. \cite{MCC} determined $\rn(G_{n,n})$, Saha and Panigrahi \cite{LP2} determined $\rn(G_{m,n})$ when $mn \equiv 0 \pmod 2$.  

Moreover, the radio number of trees has been studied extensively. Liu \cite{Daphne1} proved a general lower bound for the radio number of trees, gave a necessary and sufficient condition to achieve this bound. Later on, many families of trees have been proved to achieve this bound, including complete $m$-ary trees for $m \geq 3$ by Li et al. \cite{Li}, level-wise regular trees when all the internal vertices have degree at least 3 by H\`alasz and Tuza \cite{Halasz}, and banana trees, firecrackers trees and a special class of caterpillars by Bantva et al.  \cite{Bantva1}. 
On the other hand, there exist trees whose radio number is larger than this lower bound. For instance, odd paths \cite{Daphne2}, complete binary trees \cite{Li}, and some level-wise regular trees \cite{Daphne4}. Recently, this lower bound for those trees whose radio number does not reach the lower bound has been improved by Liu et al. \cite{Daphne3}.   

The aim of this paper is to extend the work on trees to the Cartesian product of two trees. We prove a lower bound of the radio number for these graphs and give three necessary and sufficient conditions as well as three sufficient conditions for achieving this lower bound. Applying these results, we find the radio numbers of the Cartesian products of two stars and a star with a path. 

\section{Preliminaries}
In this section, we introduce definitions and known results that will be used throughout the paper. Let $G$ be a simple finite connected  graph. For two vertices $u$ and $v$, the distance between $u$ and $v$ is the least length of a path joining $u$ and $v$, denoted by $d_G(u,v)$. If $G$ is clear in the context, we denote $d_G(u,v)$ by  $d(u,v)$. The diameter of $G$ is $\diam(G)=\max\{d(u,v) : u, v \in V(G)\}$. The \emph{weight} of a graph  $G$ on a vertex $v \in V(G)$ is defined as $w_G(v) = \sum_{u \in V(G)}d(u,v)$.  The \emph{weight of $G$} is $w(G) = \min\{w_G(v) : v \in V(G)\}$. A vertex $v \in V(G)$ is a \emph{weight center} of $G$ if $w_G(v) = w(G)$. We denote the set of weight center(s) by $W(G)$.  The following are proved in  \cite{Daphne1}.   

\begin{lemma}{\rm \cite{Daphne1}}
\label{weight1} Suppose $w$ is a weight center of a tree $T$. Then each component of $T-w$ contains at most $|V(T)|/2$ vertices. 
\end{lemma}
\begin{lemma}{\rm \cite{Daphne1}}
\label{weight2} Every tree $T$ has either one or two weight centers, and $T$ has two weight centers, say $w$ and $w'$, if and only if $ww'$ is an edge of $T$ and $T - ww'$ consists of two equal-sized components. 
\end{lemma}

Denote $P_n$ the $n$-vertex path.  It can be easily seen that  $P_n$ has two weight centers if $n$ is even (the two middle vertices), and  one weight center when $n$ is odd (the middle vertex). An  {\it $n$-star}, $n \geq 2$, denoted by $K_{1,n}$, is a tree with one vertex adjacent to $n$ leaves (degree-1 vertices). Apparently, the  non-leaf vertex is the only weight center of $K_{1,n}$. 

We view a tree $T$ rooted at its weight center(s) $W(T)$. That is, if $W(T) = \{w\}$ then $T$ is rooted at $w$; if $W(T) = \{w,w'\}$ then $T$ is rooted at both $w$ and $w'$. Let $u, v \in V(T)$. If the unique path joining $v$ and the nearest weight center to $v$ passes through $u$, then $u$ is an \emph{ancestor} of $v$, and $v$ is a  \emph{descendent} of $u$. Note that every vertex is its own ancestor and descendent. 
If $v$ is a descendent of $u$ and adjacent to $u$, then $v$ is called a \emph{child} of $u$, and $u$ the {\it parent} of $v$.  

\begin{definition}\rm
Let $u$ be a vertex adjacent to a weight center. The subtree induced by $u$ and all its descendent is called a \emph{branch} at $u$. If $u, v \not\in W(T)$ and $u, v$ are adjacent to the same weight center, then the branches induced by $u$ and $v$ are called \emph{different branches}.  If $u$ and $v$ are adjacent to different weight centers  then the branches induced by $u$ and $v$ are called \emph{opposite branches}. Note that opposite branches occur only when $|W(T)|=2$. \end{definition}

\begin{definition}\rm
In a tree $T$, we say that vertices $x$ and $y$ belong to different branches if $x$ and $y$ are in different branches, or one of them, say $x$, is a weight center and the other is in the branch induced by a vertex adjacent to $x$. We say that $x$ and $y$ are in opposite branches if $W(T) = \{w,w'\}$ and $x, y$ belong to different components of $T-ww'$. That is, $w$ and $w'$ belong to opposite branches.
\end{definition}

\begin{definition}
Let $T$ be a tree.  Define a \emph{level function} $L : V(T) \rightarrow \N \cup\{0\}$ by
$$
L_T(u) := \min\{d(u,w) : w \in W(T)\}.  
$$
The value $L_T(u)$ is called \emph{level} of $u$ in $T$. The \emph{total level of $T$} is defined as 
$$
L(T):=\sum_{u \in V(T)}L_T(u).
$$
\end{definition}
\begin{definition} \rm
Let $T$ be a tree.  For any $u, v \in V(T)$, define 
$$
\phi_T(u,v) = 
\max\{L(z) : z \ \mbox{is a common ancestor of $u$ and  $v$}
\};  
$$

$$
\delta_T(u,v) = \left\{
\begin{array}{ll}
1 & \mbox{if $W(T)=\{w, w'\}$, $u$ and $v$ belong to opposite branches;}  \\ [0.2cm]
0, & \mbox{ otherwise}.
\end{array}
\right.
$$
\end{definition}
\noindent
When it is clear in the context, we simply denote $L_T(u)$, $\phi_{T}(u,v)$, and $\delta_{T}(u,v)$, by $L(u)$, $\phi(u,v)$, and $\delta(u,v)$, respectively. 

\begin{lemma}\label{lem:phi} Let $T$ be a tree and $u,\; v \in V(T)$. The following hold:
\begin{enumerate}[\rm (a)]
\item $0 \leq \phi(u,v) \leq \max\{L(x) : x \in V(T)\}$. 
\item $\phi(u,v) = 0$ if and only if $u$ and $v$ are in different or opposite branches.  
\item The distance between two vertices $u$ and $v$ in $T$ is 
$$
d(u,v) = L(u)+L(v)+\delta(u,v)-2\phi(u,v).
$$
\end{enumerate}
\end{lemma}

\begin{definition}
Let $G_1=(V(G_1),E(G_1))$ and $G_2=(V(G_2),E(G_2))$ be two graphs. The \emph{Cartesian product} of $G_1$ and $G_2$, denoted by $G_1 \Box G_2$, is a graph with the vertex set $V(G_1 \Box G_2) = V(G_1) \times V(G_2)$, 
where two vertices $(a,b)$ and $(c,d)$ are adjacent if $a=c$ and $bd \in E(G_2)$,  
or $b=d$ and $ac \in E(G_1)$. 
\end{definition}

\begin{observation} Let $G$ and $H$ be graphs. The following hold: 
\begin{enumerate}[\rm (a)]
    \item for any $u_1,  u_2 \in V(G)$ and $v_1, v_2 \in V(H)$, $d_{G \Box H} ((u_1,v_1), (u_2, v_2)) = d_G(u_1,u_2) + d_H(v_1, v_2)$; 
    \item $\diam(G \Box H) = \diam(G)+\diam(H)$.
\end{enumerate}
\end{observation}

\begin{lemma}\label{lem:arb} 
Let $T_1$ and $T_2$ be trees. Then
$W(T_1 \Box T_2) = W(T_1) \times W(T_2)$. 
\end{lemma}

\begin{proof} Denote $G= T_1 \Box T_2$. Consider the following  cases.

\noindent
\textsf{Case-1:} $|W(T_1)| = |W(T_2)| = 1$. Denote $W(T_1) = \{w_1\}$ and $W(T_2) = \{w_2\}$. It suffices to prove that for  any $(u,v) \in V(G)$, $w_{G}((u,v)) \geq w_{G} ((w_1,w_2))$. 
Because $w_{T_1}(u) \geq w_{T_1}(w_1)$ and $w_{T_2}(v) \geq w_{T_2}(w_2)$, for any $u \in V(T_1)$ and $v \in V(T_2)$, we have 
$$
\begin{array}{llll} 
w_{G}((u,v)) 
&=& \sum_{x \in V(T_1)}d(x,u)+\sum_{y \in V(T_2)}d(y,v) \\  
&=& w_{T_1}(u)+w_{T_2}(v) \geq  w_{T_1}(w_1)+w_{T_2}(w_2) = w_{G}((w_1,w_2)). 
\end{array}
$$
\noindent
\textsf{Case-2:} $|W(T_1)| \cdot |W(T_2)| = 2$. By symmetry, assume  $W(T_1) = \{w_1\}$ and $W(T_2) = \{w_2,w'_2\}$. Then $w_{T_1}(w_1) \leq w_{T_1}(u)$ and $w_{T_2}(w_2) = w_{T_2}(w'_2) \leq w_{T_2}(v)$, for any $u\in V(T_1)$ and $v \in V(T_2)$. Let $(u,v) \in V(G)$. Then 
$$
\begin{array}{llll}
w_{G}((u,v)) 
&=& \sum_{x \in V(T_1)}d(x,u)+\sum_{y \in V(T_2)}d(y,v) \\  
&=& w_{T_1}(u)+w_{T_2}(v) \geq  w_{T_1}(w_1)+w_{T_2}(w_2) 
= w_{T_1}(w_1)+w_{T_2}(w'_2) = w_{G}((w_1,w'_2)). 
\end{array}
$$
The case for $|W(T_1)|=|W(T_2)|=2$ can be proved similarly. The proof is complete. \qed
\end{proof}

Throughout the paper, for trees $T_1$ and $T_2$, we denote $G = T_1 \Box T_2$. For any  $z_a = (x_a,y_a),\; z_b = (x_b,y_b) \in V(G)$, define: 
$$
\begin{array}{lll}
L_{G} (z) &:=& L_{T_1}(x)+L_{T_2}(y)\\
\phi_{G}((z_a,z_b)) &:=& \phi_{T_1} 
(x_a,x_b)+\phi_{T_2}(y_a,y_b),\\
\delta_{G}((z_a,z_b)) &:=& \delta_{T_1}(x_a,x_b)+\delta_{T_2}(y_a,y_b).
\end{array}
$$
\noindent 
We shall use simplified notations. For instance, we write the first in the above by $L(z) = L(x) + L(y)$, etc. The distance between two vertices  $z_a =  (x_a,y_a)$ and $z_b = (x_b,y_b)$ in  $G$ is 
\begin{eqnarray}
\label{eqn:dist3} 
d(z_a,z_b) & = & d((x_a,y_a),(x_b,y_b)) =  d(x_a,x_b)+d(y_a,y_b) \non \label{eqn:dist1} \\ 
& = & L(x_a)+L(x_b)+L(y_a)+L(y_b)+\delta(x_a,x_b)+\delta(y_a,y_b)-2\phi(x_a,x_b)-2\phi(y_a,y_b) \non \label{eqn:dist2} 
\\
& = & L(z_a)+L(z_b)+\delta(z_a,z_b)-2\phi(z_a,z_b). 
\end{eqnarray}
Each $i = 1, 2$, the tree $T_i$ has $d(w_i)$ branches if $W(T_i) = \{w_i\}$, and $d(w_i)+d(w'_i)-2$ branches if $W(T_i) = \{w_i,w'_i\}$. 
Denote the branches of $T_{1}$ and $T_2$ respectively by $T_{1,k}$ and $T_{2, k'}$, $1 \leq k \leq \beta_1$  
and $1 \leq k' \leq \beta_2$, where  
$\beta_1$ and $\beta_2$ are the numbers of branches of $T_1$ and $T_2$, respectively.  
Define a \emph{sector} $S$ of vertices of $T_1 \Box T_2$, where   
$S$ consisting of vertices $(x_a, y_b)$  so that  exactly one of the following holds for some $1 \leq k \leq \beta_1$ and $1 \leq k' \leq \beta_2$: 

$\bullet$ $x_a \in V(T_{1,k}), y_b \in V(T_{2,k'})$ \ 
$\bullet$ $x_a \in V(T_{1,k}), y_b = w_2$ \
$\bullet$ $x_a \in V(T_{1,k}), y_b = w'_2$ \

$\bullet$ $x_a = w_1, y_b \in V(T_{2,k'})$ \ \hspace{0.25in}
$\bullet$ $x_a = w'_1, y_b \in V(T_{2,k'})$ \ 
$\bullet$ $x_a \in W(T_1), y_b \in W(T_2)$.

\noindent
Totally, $G$ has $(|W(T_1)| \times  \beta_2)+(\beta_1 \times |W(T_2)|)+\beta_1\beta_2+1$ sectors. Two sectors $S_a$ and $S_b$ are different, opposite, or separate if for any $(x_a,y_a) \in S_a$ and $(x_b,y_b) \in S_b$ the following hold:    
\begin{itemize}
\item \emph{different}:   
$x_a$ and $x_b$ as well as $y_a$ and $y_b$ are in different branches of $T_1$ and $T_2$,  respectively;  
\item \emph{opposite}:  
$x_a, x_b$ are in different branches of $T_1$, and $y_a, y_b$ are in opposite  branches of $T_2$; or symmetrically  $x_a, x_b$ are in opposite branches of $T_1$ and $y_a, y_b$ are in different branches of $T_2$;  
\item \emph{separate}: 
$x_a$ and $x_b$ as well as $y_a$ and $y_b$ are in opposite branches of $T_1$ and $T_2$. 
\end{itemize}


\begin{lemma}
\label{lem:Gphi} 
Let $G = T_1 \Box T_2$ be a Cartesian product of trees $T_1$ and $T_2$, with orders $m$ and $n$ and diameter $d_1$ and $d_2$, respectively. Let $p=mn$ and $d=d_1+d_2$. Then  for any $z_a = (x_a,y_a),  \ z_b = (x_b,y_b) \in V(G)$, the following hold:
\begin{enumerate}[\rm (a)]
    \item $0 \leq \phi(z_a,z_b) \leq \max\{L(z) : z \in V(G)\}$ and $0 \leq \delta(z_a,z_b) \leq 2$;
    \item $\phi(z_a,z_b) = 0$ if and only if $z_a$ and $z_b$ are in different, opposite, or separate sectors;
    \item $\delta(z_a,z_b)= 
\left\{
\begin{array}{llll}
0, & \mbox{if $z_a$ and $z_b$ are in the same or different sectors,}  \\
1, & \mbox{if $z_a$ and $z_b$ are in opposite sectors,} \\
2, & \mbox{if $z_a$ and $z_b$ are in separate sectors.}
\end{array}
\right.
$
\end{enumerate}
\end{lemma}

\begin{lemma}
\label{lem:delp} 
Let $T_1$ and $T_2$ be trees with $|W(T_1)| = |W(T_2)| = 2$. For any ordering $\vec{V} =  (z_0, z_1,\ldots,z_{p-1})$ of $V(T_1 \Box T_2)$, the following holds: 
$$
\sum_{i=1}^{p-2} \delta(z_i, z_{i+1}) \leq 2p-3. 
$$ 
Moreover, the equality holds if and only if  $\delta(z_{(p/2)-1}, z_{p/2})=1$; and $\delta(z_t, z_{t+1})=2$ otherwise.
\end{lemma}
\begin{proof} Denote 
$
W(T_1) = \{w_1,w'_1\}, W(T_2) = \{w_2,w'_2\}, 
$ and 
$$
\begin{array}{lllll}
   W_i &=& \{v \in V(T_i): \mbox{$v$ is in the same component as $w_i$ in $T_i -w_iw'_i$} \},  &i=1,2; \\
W'_i &=&\{v' \in V(T_i): \mbox{$v'$ is in the same component as $w'_i$ in $T_i - w_iw'_i$} \}, &i=1,2.  
\end{array}
$$
By \cref{weight2}, $|W_1|=|W'_1|=m/2$ and  $|W_2|=|W'_2|= n/2$. 
Further, $V(G)$ is partitioned into four equal-size subsets denoted by: $A= W_1 \times W_2$, $B= W'_1 \times W_2$, $C=W'_1 \times W'_2$, and $D=W_1 \times W'_2$. By definition, two vertices $z_a, z_b \in V(G)$ have $\delta(z_a, z_b)=2$ if and only if $(z_a, z_b) \in (A  \times C) \cup (C \times A) \cup (B \times D) \cup (D \times B)$.  Note, $|A \times C| = |B \times D|$.  

By \cref{lem:Gphi} (a), $\sum_{i=1}^{p-2} \delta(z_i,z_{i+1}) \leq 2p-2$. Assume to the contrary,  $\sum_{i=1}^{p-2} \delta(z_i, z_{i+1}) = 2p-2$. That is, $\delta(z_i, z_{i+1})=2$ for all $0 \leq i \leq p-2$. Without loss of generality, assume $z_0 \in A$.  Since $\delta(z_0, z_1)=2$, it must be $z_1 \in C$. Similarly, as $\delta(z_1, z_2)=2$, $z_2 \in A$.  Continue this process, we have $z_i \in A$ if $i$ is even; $z_i \in C$ if $i$ is odd. This is impossible, as it covers at most half of the vertices in $G$.  Hence, the result follows. \qed
\end{proof}
\begin{definition}\rm 
Let $T_1$ and $T_2$ be trees. Two vertices $z_a$ and $z_b$ of $V(T_1 \Box T_2)$ are called {\it feasible} if $z_a$ and $z_b$ are in different sectors when $|W(T_1 \Box T_2)|=1$, opposite sectors when $|W(T_1 \Box T_2)|=2$, and separate sectors when $|W(T_1 \Box T_2)|=4$; and {\it non-feasible} otherwise. 
\end{definition} 
\begin{definition}\rm 
An ordering  $\vec{V} = (z_0, z_1, \ldots, z_{p-1})$ of $V(T_1 \Box T_2)$ is called {\it feasible} if for every  $0 \leq t \leq q-2$, $z_t$ and $z_{t+1}$ are feasible, except for the case when $|W(T_1 \Box T_2)|=4$ and $t=(p/2)-1$ for which $z_{(p/2)-1}$ and $z_{p/2}$ are in opposite sectors.  Note that if $\vec{V}$ is feasible and $|W(T_1 \Box T_2)|=4$, by  \cref{lem:delp}, $\sum_{t=1}^{p-2} \delta(z_t, z_{t+1}) = 2p-3$.
\end{definition}

%
\section{Main Results}
%

A radio labeling $f$ of a graph $G$ is an injective mapping. By shifting the labels, we may assume $f(v)=0$ for some vertex $v$. Thus, $f$ induces an ordering $\vec{V}_f = (z_0, z_1, \ldots, z_{p-1})$ of $V(G)$, where  
$$
0 = f(z_0) < f(z_1) < \ldots < f(z_{p-1}) = \span(f).
$$
 
\begin{theorem}
\label{thm:lower} 
Let $G = T_1 \Box T_2$, where $T_1$ and $T_2$ are trees with orders $m$ and $n$, and diameters $d_1$ and $d_2$, respectively. Denote $mn=p$ and $d=d_1+d_2$. Then
\begin{equation}
\label{eqn:lower}
\rn(G) \geq \left\{
\begin{array}{ll}
(p-1)(d+1)-2nL(T_1)-2mL(T_2)+1, & \mbox{ if $|W(T_1)| =|W(T_2)| = 1$}, \\ [0.2cm]
(p-1)d-2nL(T_1)-2mL(T_2), & \mbox{ if $|W(T_1)| \cdot |W(T_2)|=2$,} \\ [0.2cm]
(p-1)(d-1)-2nL(T_1)-2mL(T_2)+1, & \mbox{ if $|W(T_1)| =|W(T_2)| = 2$}. \\ [0.2cm]
\end{array}
\right.
\end{equation}
Moreover, the equality in \cref{eqn:lower} holds if and only if there exists a feasible ordering $(z_0, z_1, \ldots, z_{p-1})$ of $V(G)$ such that the following hold:
\begin{enumerate}[\rm (a)] 
\item $L(z_0)+L(z_{p-1}) = 1$ if $|W(G)|=1$, and $L(z_0)+L(z_{p-1}) = 0$ if $|W(G)| \geq 2$;
\item the following mapping $f$ is a radio labeling for $G$:  
\begin{equation}
\label{eqn:f1}
f(z_0) = 0;  f(z_{i+1}) = f(z_i)+d+1-L(z_i)-L(z_{i+1})-\delta(z_i,z_{i+1}), 0 \leq i \leq p-2.  
\end{equation}
\end{enumerate}
\end{theorem}
\begin{proof} 
Suppose that $f$ is a  radio labeling of $G$ with the ordering $\vec{V}_f = (z_0,z_1,\ldots,z_{p-1})$. By definition, $f$ satisfies the inequality $f(z_{i+1})-f(z_i) \geq d+1-d(z_i,z_{i+1})$ for all $0 \leq i \leq p-2$. Summing up these $p-1$ inequalities, 

\begin{equation}
\label{eqn:lws1}
\span(f) = f(z_{p-1}) \geq (p-1)(d+1)-\sum_{t=0}^{p-2}d(z_t,z_{t+1}).    
\end{equation} 
Denote $z_t = (x_{i_t},y_{j_t})$, $0 \leq t \leq p-1$, where $x_{i_t} \in V(T_1)$ and $y_{j_t} \in V(T_2)$. By   
Eq. (\ref{eqn:dist3}), 

\begin{eqnarray}
\sum_{t=0}^{p-2}d(z_t,z_{t+1}) & = & \sum_{t=0}^{p-2}[L(x_{i_t})+L(x_{i_{t+1}})+L(y_{j_t})+L(y_{j_{t+1}})+\delta(x_{i_t},x_{i_{t+1}})+\delta(y_{j_t},y_{j_{t+1}}) \non \\
& & -2\phi(x_{i_t},x_{i_{t+1}})-2\phi(y_{j_t},y_{j_{t+1}})] \non \\
& = & 2n L(T_1)+2mL(T_2)-L(x_{i_0})-L(x_{i_{p-1}})-L(y_{j_0})-L(y_{j_{p-1}}) \non \\
& & + \sum_{t=0}^{p-2}\left[\delta(z_{t}, z_{t+1}) -2\phi(x_{i_t},x_{i_{t+1}})-2\phi(y_{j_t},y_{j_{t+1}})\right]
\label{eqn:lws2}
\end{eqnarray}

We proceed the proof by three cases. 

\noindent
\textsf{Case-1:} $|W(T_1)| = |W(T_2)| = 1$.  Then  $\delta(x_{i_t},x_{i_{t+1}}) = \delta(y_{j_t},y_{j_{t+1}}) = 0$ for all $i, j$. Because  $|W(G)|=1$, so  $L(x_{i_0})+L(x_{i_{p-1}})+L(y_{j_0})+L(y_{j_{p-1}}) \geq 1$. 
As $\phi(x,y) \geq 0$, for any $x, y$:   
$$
\sum_{t=0}^{p-2}d(z_t,z_{t+1}) 
 \leq  2nL(T_1)+2mL(T_2)-1.
$$
Substituting the above to  \cref{eqn:lws1}, we obtain
$$
 \span(f) = f(z_{p-1}) \geq (p-1)(d+1)-2nL(T_1)-2mL(T_2)+1.  
$$

\noindent
\textsf{Case-2:} $|W(T_1)| \cdot |W(T_2)| = 2$.  By symmetry, assume $|W(T_1)| = 1$ and $|W(T_2)| = 2$. In this case,  $\delta(x_{i_t},x_{i_{t+1}}) = 0$ and $0 \leq \delta(y_{j_t},y_{j_{t+1}}) \leq 1$ for $0 \leq t \leq p-2$. Hence $\delta(z_i, z_{i+1}) \leq 1$ for all $i$.  Further, since $|W(G)|=2$, it holds that  $L(x_{i_0})+L(x_{i_{p-1}})+L(y_{j_0})+L(y_{j_{p-1}}) \geq 0$. Hence, 
$$
\sum_{t=0}^{p-2}d(z_t, z_{t+1}) 
 \leq  2 n L(T_1)+2 m L(T_2)+(p-1).
$$
Similarly, the result follows by substituting the above into \cref{eqn:lws1}. 

\noindent
\textsf{Case-3:} $|W(T_1)|=|W(T_2)|=2$. As $G$ has four weight centers, it holds that $L(x_{i_0})+L(x_{i_{p-1}})+L(y_{j_0})+L(y_{j_{p-1}}) \geq 0$. By \cref{lem:delp}, $\sum_{i=0}^{p-1} \delta(z_i, z_{i+1}) \leq 2p-3$. Hence 
$$
\sum_{t=0}^{p-2}d(z_t,z_{t+1}) \leq  2nL(T_1)+2mL(T_2)+2p-3.
$$
The results follows by substituting the above into \cref{eqn:lws1}. 

It is clear that the equality holds in \cref{eqn:lower} if and only if all the qualities hold in \cref{def:rn} as well as the ones in the above, which are reflected on the moreover part.  \qed
\end{proof}

It is easy to see that $\rn(P_2 \Box P_2)$ is equal to the lower bound in \cref{eqn:lower}. Indeed, this is the only possibility that the lower bound of $\rn(T_1 \Box T_2)$ in  \cref{thm:lower} is sharp when $|W(T_1 \Box T_2)| = 4$. 

\begin{theorem}
\label{thm:slack} 
Let $T_1$ and $T_2$ be trees of orders $m, n$ and diameters $d_1, d_2$, respectively. Suppose $|W(T_1)| = |W(T_2)|= 2$. Then the equality in  \cref{eqn:lower} holds only if $T_1=T_2 = P_2$. 
\end{theorem}
\begin{proof} Let $T_1$ and $T_2$ be trees that satisfy the hypotheses. Denote $G = T_1 \Box T_2$. Since $|W(T_1)|=|W(T_2)|=2$,  $m$ and $n$ are even, so  $p$ is  even.  Assume to the contrary that it is not the case that $T_1=T_2=P_2$ but $\rn(G) = (p-1)(d-1)-2nL(T_1)-2mL(T_2)+1$. By  \cref{thm:lower}, there exists a feasible ordering $(z_0,z_1,\ldots,z_{p-1})$ of $V(G)$ such that Theorem \ref{thm:lower} (a) (b) are satisfied. Since  $\diam(G) = d \geq 4$ (as $G \neq P_2 \Box P_2$) and $d(w,w') \leq 2$ for any $w,w' \in W(G)$,   
there exist vertices $z_x$ and $z_y$ in separate or opposite sectors where $L(z_{x}), L(z_{y}) \geq \frac{d-2}{2}$ and $\{x, y\} \neq \{\frac{p}{2} - 2, \frac{p}{2}\}$.  Hence, there exists $z_t \in V(G)$ such that $L(z_t) \geq \frac{d-2}{2}$ and $t-1,  t+1 \neq\frac{p}{2}-1$. 
Consider $z_{t-1}$ and $z_{t+1}$ for \cref{eqn:f1}. By Theorem \ref{thm:lower} (b) and $\delta(z_{t-1}, z_{t}) = \delta (z_{t}, z_{t+1})=2$, we arrive at the following contradiction:  
$$
\begin{array}{lll}
f(z_{t+1})-f(z_{t-1}) &=&  2(d+1)-4-L(z_{t-1})-2L(z_t)-L(z_{t+1}) \\
&\leq&  2(d+1)-4-L(z_{t-1})-2(\frac{d-2}{2})-L(z_{t+1}) \\
&=& d-L(z_{t-1})-L(z_{t+1}) \\
&<& 
 d+1-d(z_{t-1},z_{t+1}) \ \ \ \mbox{($\because d(z_{t-1},z_{t+1}) \leq L(z_{t-1})+L(z_{t+1}))$}. \hspace{0.6in}\Box \\
\end{array}.  
$$
\end{proof}

In the next two results, we give additional necessary and sufficient conditions for $T_1 \Box T_2$ to achieve the lower bound in \cref{thm:lower}. 

\begin{theorem}
\label{thm:upper1} 
Let $G = T_1 \Box T_2$, where $T_1$ and $T_2$ are trees with orders $m$ and $n$, and diameters $d_1$ and $d_2$, respectively. Let $p=mn$ and $d=d_1 + d_2$. Then the equality of \cref{eqn:lower} 
holds if and only if there exists a feasible  ordering $\vec{V}=(z_0, \ldots,z_{p-1})$ of $V(G)$ such that the following hold: 
\begin{enumerate}[\rm (a)]
\item $L(z_0)+L(z_{p-1})=1$ when $|W(G)|=1$,  and $L(z_0)+L(z_{p-1})=0$ when $|W(G)| \geq 2$;
\item for any two vertices $z_a, z_b\;(0 \leq a < b \leq p-1)$ the following is satisfied
\begin{equation}
\label{eqn:dab}
d(z_a,z_b) \geq \sum_{t=a}^{b-1} \left[ 
L(z_t)+L(z_{t+1})+\delta(z_t, z_{t+1})\right] - (b-a-1) (d+1).  
\end{equation}
\end{enumerate}
\end{theorem} 

\begin{proof} 
\textsf{Necessity:} Suppose that the equality of \cref{eqn:lower} holds.  By \cref{thm:lower}, there exists an optimal radio labeling $f$ of $G$ with a feasible ordering $\vec{V}$ of $V(G)$  
such that (a) and (b) in \cref{thm:lower} hold. 
For any two vertices $z_a$ and  $z_b\;(0 \leq a < b \leq p-1)$ in  $\vec{V}$, 
we have
\begin{equation*}
f(z_b)-f(z_a) = \sum_{t=a}^{b-1}\left[
d+1-L(z_t)-L(z_{t+1})-\delta(z_t,z_{t+1})
\right].
\end{equation*}
Since $f$ is a radio labeling of $G$,  $f(z_b)-f(z_a) \geq d+1-d(z_a,z_b)$. Substituting this to the above, \cref{eqn:dab} is obtained.  Hence (b) is true. 

\textsf{Sufficiency:} Let $\vec{V}=(z_0,z_1,\ldots,z_{p-1})$ be an ordering satisfying (a) and (b). It is clear that $\vec{V}$ is feasible. By \cref{thm:lower} it  suffices to prove that the mapping $f$ defined by \cref{eqn:f1} on $\vec{V}$ is a radio labeling. Let $z_a$ and $z_b\;(0 \leq a < b \leq p-1)$ be two arbitrary vertices. By  \cref{eqn:f1} and (b), we obtain
$$
f(z_b)-f(z_a)  =  (b-a)(d+1)-\sum_{t=a}^{b-1}\left[L(z_t)+L(z_{t+1})+\delta(z_t, z_{t+1})\right]
 \geq  d+1-d(z_a,z_b).
$$
Hence $f$ is a radio labeling for $G$. The proof is complete. \qed
\end{proof}

For integers $x \leq y$, denote $[x,y]=\{x, x+1, \cdots, y\}$. 

\begin{theorem}
\label{thm:upper2} 
Let $G = T_1 \Box T_2$, where $T_1$ and $T_2$ are trees with orders $m$ and $n$, and diameters $d_1 \geq 2$ and $d_2 \geq 2$,  respectively. Let $p=mn$ and $d=d_1+d_2$. Denote $\xi= |W(T_1)|+|W(T_2)| - 2$. Then the equality of \cref{eqn:lower}  
holds if and only if there exists a feasible  order  $\vec{V}=(z_0,z_1,\ldots,z_{p-1})$ of $V(G)$ such that all the following hold: 
\begin{enumerate}[\rm (a)]
\item $L(z_0)+L(z_{p-1})=1$ when $|W(G)|=1$,  and $L(z_0)+L(z_{p-1})=0$ when $|W(G)| \geq 2$;
\item $L(z_s) \leq \frac{d+1-2\xi}{2}$ for all $0 \leq s \leq p-1$, except when  $|W(G)|=4$ and $s \in \{ (p/2)-1, p/2\}$, for which  $L(z_{s}) \leq \frac{d+3-2\xi}{2}$;
\item any two non-feasible vertices $z_a$ and $z_b\;(0 \leq a < b \leq p-1)$ satisfy
$$
\phi(z_a,z_b) \leq 
\left\{
\begin{array}{lll}
(b-a-1)\left(\frac{d+1-\xi}{2}\right)-\sum\limits_{t=a+1}^{b-1}L(z_t)-\left(\frac{\xi-\delta(z_a,z_b)-1}{2}\right),     &\mbox{if $|W(G)|=4$ and}\\ 
&\mbox{$\{(p/2)-1,p/2\} \subseteq [a,b]$}, \\
(b-a-1)\left(\frac{d+1-\xi}{2}\right)-\sum\limits_{t=a+1}^{b-1}L(z_t)-\left(\frac{\xi-\delta(z_a,z_b)}{2}\right),     &\mbox{otherwise.} 
\end{array}
\right. 
$$
\end{enumerate}
Moreover, with conditions (a), (b), (c),  the mapping defined by \cref{eqn:f1} is an optimal radio labeling of $G$.
\end{theorem}

\begin{proof}\textsf{Necessity:} Suppose the equality of \cref{eqn:lower} holds. By Theorem  \ref{thm:upper1}, there exists a feasible order $\vec{V} = (z_0,z_1,\ldots,z_{p-1})$ of $V(G)$ such that \cref{thm:upper1} (a) (b) hold. By \cref{thm:upper1} (a),   
$L(z_0), L(z_{p-1}) \leq 1$. If $d_1 = d_2 = 2$, then $|W(T_1)| = |W(T_2)| = 1$  and $\xi = 0$. So, $L(z_0),L(z_{p-1}) \leq 1 \leq (d+1-2\xi)/2$. These inequalities also hold for other cases.

For $1 \leq s \leq (p-2)$, applying \cref{eqn:dab} to $d(z_{s-1},z_{s+1})$ and by Eq. \eqref{eqn:dist3}, we obtain
\begin{equation}\label{eqn:maxl}
2L(z_s) \leq d+1-\delta(z_{s-1},z_s)-\delta(z_s,z_{s+1})+\delta(z_{s- 1},z_{s+1})-2\phi(z_{s-1},z_{s+1}).
\end{equation}
Since $\vec{V}$ is feasible, $\delta(z_s,z_{s+1}) = \xi$ and $\delta(z_{s-1},z_{s+1}) = 0$ 
for all $0 \leq s \leq p-2$,  except when  $|W(G)|=4$ and $s \in \{(p/2)-1, p/2\}$.  Substituting this (without the exceptional case) into \cref{eqn:maxl}, we have 
$
2L(z_s) \leq d+1-2\xi.
$
When $|W(G)|=4$ and $s \in \{(p/2)-1, p/2\}$, we have   $\delta(z_{s-1},z_{s})+\delta(z_s,z_{s+1}) = 2\xi-1=3$ and $\delta(z_{s-1},z_{s+1})=1$. Substituting this into \cref{eqn:maxl}, we get  
$
2L(z_s) \leq d+3-2\xi.
$
Thus, (b) is true. 

To prove (c), assume $z_a$ and $z_b$ are non-feasible vertices. Combining Eq. \eqref{eqn:dist3} and \cref{eqn:dab}, 
$$
\delta(z_a,z_b)-2\phi(z_a,z_b) \geq 2\sum_{t=a+1}^{b-1}L(z_t)+\sum_{t=a}^{b-1}\delta(z_t,z_{t+1})-(b-a-1)(d+1).
$$
Since $\vec{V}$ is feasible,  $\delta(z_t,z_{t+1}) = \xi$ for all $0 \leq t \leq p-2$, except when $|W(G)|=4$ and $t=(p/2)-1$, in which  $\delta(z_{(p/2)-1}, z_{p/2}) = \xi-1$. Substituting all these into the  above, (c) is satisfied. 

\textsf{Sufficiency:} Suppose there exists a feasible order $\vec{V} =  (z_0,z_1,\ldots,z_{p-1})$ of $V(G)$ such that (a)-(c) hold. It suffices to prove that \cref{eqn:dab} in Theorem \ref{thm:upper1} is true for any $z_a$ and $z_b$, $0 \leq a < b \leq p-1$. Denote the right-hand side of \eqref{eqn:dab} by $S_{a,b}$. 

Assume $z_a$ and $z_b$ are feasible vertices. Then $d(z_a, z_b) =  L(z_a)+L(z_b)+\delta(z_a,z_b)$. As $\vec{V}$ is feasible,  $\delta(z_a,z_b) =  \delta(z_a, z_{a+1})$. When $|W(G)|=4$, $\delta(z_a, z_{a+1})=1$  occurs only when $a = (p/2)-1$. In this case, $\delta(z_a, z_b)=1$ for all $b \geq a$. Thus, it also holds that $\delta(z_a,z_b) =  \delta(z_a, z_{a+1})$. 

By (b), $2L(z_t)+\delta(z_t,z_{t+1})-(d+1) \leq 0$. Hence 
$$
\begin{array}{lll}
S_{a,b} &=& L(z_a)+L(z_b)+\delta(z_a,z_{a+1})+\sum\limits_{t=a+1}^{b-1}
\left[2L(z_t)+\delta(z_t,z_{t+1})-(d+1)\right] \\
&\leq& L(z_a)+L(z_b)+\delta(z_a,z_{a+1}) = L(z_a)+L(z_b)+\delta(z_a,z_b) = d(z_a,z_b).
\end{array}
$$

If $z_a$ and $z_b$ are non-feasible vertices, as  $\vec{V}$ is feasible, \cref{eqn:dab} can be obtained by (c). The proof is complete. \qed
\end{proof}

In the next result, we give three  sufficient conditions for the lower bound given in \cref{eqn:lower} to be tight when $|W(T_1)| \cdot |W(T_2)| \leq 2$.  

\begin{theorem}\label{thm:suff} 
Let $G=T_1 \Box T_2$, where $T_1$ and $T_2$ are trees of orders $m$ and $n$, and diameters  $d_1$ and $d_2$, respectively, and $|W(T_1)| \cdot |W(T_2)| \leq 2$. Denote  $p=mn$, $d=d_1+d_2$, and $\xi =|W(T_1)|+|W(T_2)| - 2$. Then equality of \cref{eqn:lower} holds if there exists a feasible order  $\vec{V}=(z_0,z_1,\ldots,z_{p-1})$ of $V(G)$ such that the following are true: 
\begin{enumerate}[\rm (a)]
\item $L(z_0)+L(z_{p-1})=1$ when $|W(G)|=1$, and $L(z_0)+L(z_{p-1})=0$ when $|W(G)| \geq 2$, 
\item any of the following conditions holds:
\begin{enumerate}[\rm (i)]
\item $\min\{d(z_t,z_{t+1}),d(z_{t+1},z_{t+2})\} \leq \frac{d+\xi}{2}$, for all $0 \leq t \leq p-3$,
\item $d(z_t,z_{t+1}) \leq \frac{d+\xi+2}{2}$, for all $0 \leq t \leq p-2$, 
\item $L(z_s) \leq \frac{d+1-\xi}{2}$, for all $0 \leq s \leq p-1$; and if $b-a < d$ then $z_a$ and $z_b$ are in different or opposite sectors.
\end{enumerate}
\end{enumerate}
\end{theorem}
\begin{proof} We prove that if there exists a feasible order  $\vec{V}=(z_0,z_1,\ldots,z_{p-1})$ of $V(G)$ such that (a) and one of (b.i), (b.ii), (b.iii) holds, then $\vec{V}$ satisfies Theorem \ref{thm:upper1} (a) (b). Since  (a) is the same as  \cref{thm:upper1} (a), it  suffices to prove that Theorem \ref{thm:upper1} (b) is true for any two vertices $z_a$ and $z_b$, $0 \leq a < b \leq p-1$. Denote the right-hand side of \cref{eqn:dab} by $S_{a,b}$. 

\noindent
\textsf{Case-1:}  $|W(T_1)|=|W(T_2)|=1$. In this case, $\xi = 0$ and $\delta(z_t,z_{t+1}) = 0$ for all $0 \leq t \leq p-1$. 

\textsf{Subcase 1.1:} Suppose (a) and (b.i) hold.  By (b.i) and as $\vec{V}$ is feasible, for each $1 \leq t \leq p-2$, $L(z_t) \leq \min\{L(z_{t-1})+L(z_t), L(z_t)+L(z_{t+1})\} = \min\{d(z_{t-1},z_t),d(z_t,z_{t+1})\} \leq d/2$. 

Let $z_a$ and $z_b$ be two arbitrary vertices. If $z_a$ and $z_b$ are feasible, then  $d(z_a,z_b) = L(z_a)+L(z_b)$,  
$$
\begin{array}{lll} 
S_{a,b} &=& L(z_a)+L(z_b)+2 \left[ \sum\limits_{t=a+1}^{b-1}L(z_t)\right] -(b-a-1)(d+1) \\
&\leq& L(z_a)+L(z_b)-(b-a-1) 
\leq L(z_a)+L(z_b)
= d(z_a,z_b).
\end{array}
$$

Assume $z_a$ and $z_b$ are non-feasible. Then  $d(z_a,z_b) = L(z_a)+L(z_b)-2\phi(z_a,z_b)$.  If $b-a \geq 4$, by (b.i), $d(z_{a+1}, z_{a+2}) \leq d/2$ or $d(z_{a+2}, z_{a+3}) \leq d/2$.  
In either case, $2[(L(z_{a+1})+ L(z_{a+2}) + L(z_{a+3})] \leq 2d$. Thus, 
$$
\begin{array}{llll}
S_{a,b} &\leq& L(z_a) + L(z_b) + 2[(L(z_{a+1})+ L(z_{a+2}) + L(z_{a+3})]-3(d+1) \\
&\leq&  L(z_a)+L(z_b)-d-3 \\
&=& L(z_a) + L(z_b) - 2((d+3)/2) \\
&\leq& L(z_a) + L(z_b) - 2 \phi(z_a, z_b) \ \  (\because \phi(z_a, z_b) \leq \min\{L(z_a), L(z_b)\}).
\end{array} 
$$

Suppose $b - a = 3$. Assume $\max \{d(z_a,z_{a+1}), d(z_{a+1},z_{a+2})\} \leq d/2$. Similarly to the above, we get $S_{a,b} \leq L(z_a)+L(z_b)-d-2 = L(z_a)+L(z_b)-2((d+2)/2) \leq L(z_a)+L(z_b)-2\phi(z_a,z_b) = d(z_a,z_b)$. Assume $\max\{d(z_a,z_{a+1}),d(z_{a+1},z_{a+2})\} > d/2$. Without loss of generality, suppose  $d(z_a,z_{a+1}) < d(z_{a+1},z_{a+2})$. Recall  $L(z_{a+2}) \leq d/2$. By (b.i), $L(z_{a+1}) \leq d/2-L(z_a)$, implying 
$$
\begin{array}{llll} 
S_{a,b} &=& L(z_a) + L(z_{a+3}) + 2 [L(z_{a+1}) + L(z_{a+2})] - 2(d+1)\\
&\leq& L(z_a)+L(z_{a+3})-2L(z_a) \\
&\leq& L(z_a)+L(z_{a+3})-2\phi(z_a,z_{a+3}) = d(z_a,z_{b}).
\end{array}
$$
Finally, assume $b-a=2$.  By (b.i), $L(z_a)+L(z_{a+1}) \leq d/2$ or $L(z_{a+1})+L(z_{a+2}) \leq d/2$. Without loss of generality, assume  $L(z_a)+L(z_{a+1}) \leq d/2$, implying $S_{a,b} \leq L(z_a)+L(z_{a+2})-2(L(z_a)+1)$ $\leq L(z_a)+L(z_{a+2})-2\phi(z_a,z_{a+2}) = d(z_a,z_{a+2})$.      

\textsf{Subcase 1.2:} Suppose (a) and (b.ii) hold. If $b-a = 1$, then $z_a$ and $z_b$ are feasible vertices.  Hence, $S_{a,b} =  L(z_a)+L(z_b) = d(z_a,z_b)$, so  \cref{eqn:dab} is satisfied. If $b-a \geq 2$, by (b.ii),  $d(z_t,z_{t+1}) = L(z_t)+L(z_{t+1}) \leq (d+2)/2$ for $a \leq t \leq b-1$. Hence 
$$
S_{a,b} \leq \sum_{t=a}^{b-1}((d+2)/2)-(b-a-1)(d+1) \leq 1 \leq d(z_a,z_b).
$$
Thus \cref{eqn:dab} is satisfied.   

\textsf{Subcase 1.3:} Suppose (a) and (b.iii) hold. Assume $z_a$ and $z_b$ are feasible vertices. If $b = a+1$, the result follows.  Assume $b \geq a+2$. Then $d(z_a,z_b) = L(z_a)+L(z_b)$ and $S_{a,b} = L(z_a)+L(z_b)+2\sum_{t=a+1}^{b-1}L(z_t)-(b-a-1)(d+1) \leq L(z_a)+L(z_b)+2((d+1)/2)-(b-a-1)(d+1) \leq L(z_a)+L(z_b) = d(z_a,z_b)$. 

If $z_a$ and $z_b$ are not feasible, then $d(z_a,z_b) = L(z_a)+L(z_b)-2\phi(z_a,z_b)$. Assume $d$ is even. By (b.iii), $L(z_t) \leq d/2$ for $0 \leq t \leq p-1$. By the second part of (b.iii),  $S_{a,b} 
\leq L(z_a)+L(z_b)-(b-a-1) \leq  L(z_a)+L(z_b)-2((d-1)/2) 
\leq L(z_a)+L(z_b)-2\phi(z_a,z_b) = d(z_a,z_b)$.  Assume $d$ is odd. By (b.iii),  $L(z_t) \leq (d+1)/2$ for $0 \leq t \leq p-1$. As $\max\{L(z_t)+L(z_{t+1}) : 0 \leq t \leq p-2\} \leq d$, 
$$
2\sum_{t=a+1}^{b-1}L(z_t) \leq (b-a-1)\left[(d+1)/2\right]+(b-a-1)\left[(d-1)/2\right].
$$
Hence, we have
$$
\begin{array}{llll} 
S_{a,b} &=& L(z_a)+L(z_b)+2 \left[ \sum\limits_{t=a+1}^{b-1}L(z_t)\right]-(b-a-1)(d+1) \\
&\leq& L(z_a)+L(z_b)+(b-a-1)\left[(d+1)/2\right]+(b-a-1)\left[(d-1)/2)\right]-(b-a-1)(d+1) \\
&=& L(z_a)+L(z_b)-(b-a-1) \\
&\leq& L(z_a)+L(z_b)-2((d-1)/2) \ \ \  \mbox{(by the second part of (b.iii))} \\
&\leq& L(z_a)+L(z_b)-2\phi(z_a,z_b) = d(z_a,z_b).
\end{array}
$$

\noindent
\textsf{Case-2:} $|W(T_1)|\cdot|W(T_2)| = 2$. Then $\xi = 1$.  Without loss of generality, assume  $|W(T_1)| = 1$ and $|W(T_2)| = 2$. As $\vec{V}$ is feasible, $\delta(z_t,z_{t+1}) = 1$ for all $0 \leq t \leq p-1$.  

\textsf{Subcase 2.1:} Suppose (a) and (b.i) hold. Since $\vec{V}$ is feasible, $d(z_{t},z_{t+1}) = L(z_t) + L(z_{t+1})+1$ for all $t$.  By (b.i), $L(z_t) \leq \min\{L(z_{t-1})+L(z_t), L(z_t)+L(z_{t+1})\} = \min\{d(z_{t-1},z_t)-1, d(z_t,z_{t+1})-1\} \leq (d-1)/2$ for  $1 \leq t \leq p-2$. 

Let $z_a$ and $z_b$ be  vertices. If $z_a$ and $z_b$ are feasible, then $d(z_a,z_b)=L(z_a)+L(z_b)+1$. Thus,
$$
\begin{array}{llll}
S_{a,b} &=&L(z_a)+L(z_b)+1+\sum\limits_{t=a+1}^{b-1}(2L(z_t)+1)-(b-a-1)(d+1) \\
&\leq& L(z_a)+L(z_b)+1-(b-a-1) 
\leq L(z_a)+L(z_b)+1 = d(z_a,z_b). 
\end{array}
$$
If $z_a$ and $z_b$ are non-feasible,  then $d(z_a,z_b) = L(z_a)+L(z_b)-2\phi(z_a,z_b)$. If $b-a \geq 4$, then by (b.i), $L(z_{a}) + L(z_{a+1}) \leq (d-1)/2 $ or 
$L(z_{a+1}) + L(z_{a+2}) \leq (d-1)/2$.  This implies,  $2[L(z_{a+1}) + L(z_{a+2}) + L(z_{a+3})] \leq 2(d-1)$. Hence, $S_{a,b} \leq L(z_a)+L(z_b)-(d+1)   
\leq L(z_a)+L(z_b)-2\phi(z_a,z_b)=d(z_a,z_b)$. 

Suppose $b-a=3$. If   $\max\{d(z_a,z_{a+1}), d(z_{a+1},z_{a+2})\} \leq (d+1)/2$, then $L(z_{a+1}) + L(z_{a+2}) \leq (d-1)/2$. Hence, 
$$
\begin{array}{llll}
S_{a,b} &=&
L(z_a)+L(z_b)+3+2[L(z_{a+1})+L(z_{a+2})]-2(d+1)  \\
&\leq& L(z_a)+L(z_b)-d 
\leq L(z_a)+L(z_b)-2\phi(z_a,z_b) =d(z_a,z_b). 
\end{array} 
$$
Assume $\max\{d(z_a,z_{a+1}), d(z_{a+1},z_{a+2})\} > (d+1)/2$.  Without loss of generality, let  $d(z_a,z_{a+1}) < d(z_{a+1},z_{a+2})$. By (b.i),  $L(z_{a+1}) \leq (d-1)/2-L(z_a)$. Hence, 
$$S_{a,b} \leq L(z_a)+L(z_b)-2L(z_a)-1 \leq 
L(z_a)+L(z_b)-2\phi(z_a,z_b) = d(z_a,z_b).$$

Finally, assume $b-a=2$. Then $L(z_a)+L(z_{a+1}) \leq (d-1)/2$ or $L(z_{a+1})+L(z_{a+2}) \leq (d-1)/2$. Without loss of generality, assume  $L(z_a)+L(z_{a+1}) \leq (d-1)/2$. Then 
$$
S_{a,b} \leq L(z_a)+L(z_b)-2L(z_a) \leq L(z_a)+L(z_b)-2\phi(z_a,z_b) = d(z_a,z_b).
$$

\textsf{Subcase-2.2:} Suppose (a) and (b.ii) hold. If $b-a=1$, then $z_a$ and $z_b$ are feasible vertices. So $S_{a,b} = L(z_a)+L(z_b)+1 = d(z_a,z_b)$. If $b-a \geq 3$, by (b.ii), $d(z_t,z_{t+1}) = L(z_t)+L(z_{t+1})+1 \leq (d+3)/2$ for all $0 \leq t \leq b-1$. Hence $S_{a,b} \leq \sum_{t=a}^{b-1}((d+3)/2)-(b-a-1)(d+1) \leq (3(d+3)/2)-2(d+1) \leq 0 < d(z_a,z_b)$. 

Assume $b-a = 2$. If $d(z_a, z_{a+2}) \geq 2$, then $S_{a,a+2} \leq (d+3) - (d+1) = 2 \leq d(z_a, z_{a+2})$. If $d(z_a, z_{a+2})=1$, as  $z_a=(x_{i_a}, y_{j_{a}})$ and $z_{a+2}=(x_{i_{a+2}},  y_{j_{a+2}})$, either $x_{i_a} = x_{i_{a+2}}$ and $d_{T_2} (y_{j_a},y_{j_{a+2}})=1$, or $y_{j_a} = y_{j_{a+2}}$ and $d_{T_1}(x_{i_{a}}, x_{i_{a+2}})=1$. Assume $x_{i_a} = x_{i_{a+2}}$ and $d_{T_2} (y_{j_a},y_{j_{a+2}})=1$. Then direct calculation shows that  $|d(z_a,z_{a+1}) -  d(z_{a+1},z_{a+2})|=1$.  
Combining with (b.ii), $\min\{d(z_a,z_
{a+1}), d(z_{a+1}, z_{a+2})\} \leq (d+1)/2$, which   
implies, $S_{a,a+2} \leq (d+3)/2+(d+1)/2-(d+1) = 1 = d(z_a,z_b)$. Thus, \cref{eqn:dab} is satisfied.

\textsf{Subcase-2.3:} Suppose (a) and (b.iii) hold. If $z_a$ and $z_b$ are feasible vertices then $d(z_a,z_b) = L(z_a)+L(z_b)+1$ and $S_{a,b} = L(z_a)+L(z_b)+1+\sum_{t=a+1}^{b-1}(2L(z_t)+1)-(b-a-1)(d+1) \leq L(z_a)+L(z_b)+1+\sum_{t=a+1}^{b-1}(2(d/2)+1)-(b-a-1)(d+1) = L(z_a)+L(z_b)+1 = d(z_a,z_b)$. If $z_a$ and $z_b$ are non-feasible then $d(z_a,z_b) = L(z_a)+L(z_b)-2\phi(z_a,z_b)$. Assume $d$ is even. Then $L(z_t) \leq d/2$ for all $t \in [0, p-1]$. Since $\max\{L(z_t)+L(z_{t+1}) : t \in [0, p-2]\} = \max\{d(z_t,z_{t+1})-1 : t \in [0, p-2]\} \leq d-1$, it is impossible that $L(z_t)=L(z_{t+1})=d/2$. Thus,
$$
\sum_{t=a+1}^{b-1}(2L(z_t)+1) \leq \frac{(b-a-1)(d+1)}{2}+\frac{(b-a-1)(d-1)}{2} = (b-a-1)d.
$$
Hence 
$$
\begin{array}{llll} 
S_{a,b} &=& L(z_a)+L(z_b)+\sum\limits_{t=a+1}^{b-1}\left[2L(z_t)+1\right]-(b-a-1)(d+1) \\ 
&\leq& L(z_a)+L(z_b)-(b-a-1) \\
&\leq& L(z_a)+L(z_b)-(d-1) 
\ \ \  \mbox{(by the second part of (b.iii))} 
\\
& \leq & L(z_a)+L(z_b)-2\left[(d-1)/2\right] \\
&\leq & L(z_a)+L(z_b)-2\phi(z_a,z_b)=d(z_a,z_b).
\end{array} 
$$

Suppose $d$ is odd. By (b.iii), $L(z_t) \leq (d-1)/2$ for all $t \in [0, p-1]$. Similar to the above, we get $S_{a,b} \leq
L(z_a)+L(z_b)-(b-a-1)$. The result follows by the same calculation above.
\qed
\end{proof}

\section{Sharpness of the Bounds}
In this section, we determine the radio number of some Cartesian products of two trees using Theorems \ref{thm:lower} to \ref{thm:upper2}. 

\begin{theorem}\label{thm:star} Let $m \geq n \geq 3$ be integers. Then $\rn(K_{1,m} \Box K_{1,n}) = mn+3(m+n)+1$.
\end{theorem}
\begin{proof} Denote $G = K_{1,m} \Box K_{1,n}$.  Then $p=|E(G)|=(m+1)(n+1)$, $\diam(G) = 4$, $|W(G)|=1$, $L(K_{1,m}) = m$, and $L(K_{1,n}) = n$. Substituting these into \cref{eqn:lower}, we obtain the lower bound for $\rn(G)$. We prove that the bound is attained by giving a feasible ordering $\vec{V}=(z_0,z_1,\ldots,z_{p-1})$ of $V(G)$ satisfying the conditions of Theorem \ref{thm:upper1}. 

Denote $V(K_{1,m}) = \{x_0,x_1,\ldots,x_{m}\}$,  $E(K_{1,m}) = \{x_0x_i : 1 \leq i \leq m\}$, $V(K_{1,n}) = \{y_0,y_1,\ldots, y_n\}$, and  $E(K_{1,n}) = \{y_0y_i : 1 \leq i \leq n\}$. Define an ordering of $V(G)$ as follows: For $(i, j) \in [0, m] \times [0,n]$, let $z_t = (x_i,y_j)$, where
\begin{equation*}
t := \left\{
\begin{array}{ll}
(i-j)(n+1)+j, & \mbox{ if $i \geq j$}, \\ [0.2cm]
(m+2+i-j)(n+1)-i-1, & \mbox{ if $i < j$}.
\end{array}
\right.
\end{equation*}

See \cref{Fig1} for an example. It is easy to see that the given ordering is 
feasible and $L(z_0)+L(z_{p-1})=1$. It suffices to prove $\vec{V}$ satisfies \cref{thm:upper1} (b), by showing that \cref{eqn:dab} is satisfied. Let $z_a$ and $z_b$ be two arbitrary vertices, $0 \leq a < b \leq p-1$. Denote the right-hand side of \eqref{eqn:dab} by $S_{a,b}$. Since $\diam(G) = 4$, and for all $0 \leq t \leq p-2$, $L(z_t) \leq 2$ and  $\delta(z_t,z_{t+1})=0$, we have 
$$
S_{a,b} = \sum\limits_{t=a}^{b-1}(L(z_t)+L(z_{t+1}))-5(b-a-1) \leq 
4(b-a)-5(b-a-1)  = 5 - (b-a).  
$$
If $b-a \geq 4$, then $S_{a,b} \leq 1 \leq d(z_a,z_b)$. If  $b-a = 3$, then our labeling ordering has $d(z_a, z_b) \geq 2$, hence  $S_{a,b} \leq 2 \leq d(z_a,z_b)$. 

Assume $b-a = 2$. Let $C = \{z_t = (x_{i_t},y_{j_t}) : x_{i_t} = x_0 \mbox{ or } y_{j_t} = y_0\}$. If $z_a = z_0 \in C$, then $d(z_0, z_2)=2$ and $S_{a,b} = 1 < d(z_0,z_2)$. If $z_a \neq z_0$, then  either $|\{z_a, z_b\} \cap C|=1$ or $|\{z_a,z_{a+1},z_b\} \cap C| = 0$. The former has $d(z_a,z_b) \geq 3$, the latter has $d(z_a,z_b) \geq 4$, and for both $S_{a,b} \leq d(z_a, z_b)$. If  $b-a=1$, since $\vec{V}$ is feasible, it holds $d(z_a, z_b) = L(z_a)+L(z_b)$. Hence, $S_{a,b} = L(z_a) + L(z_b) = d(z_a,z_b)$. 
Therefore, \cref{eqn:dab} is satisfied. The proof is complete.  \qed
\end{proof}
\begin{figure}[ht!]
    \centering
    \includegraphics[width=3in]{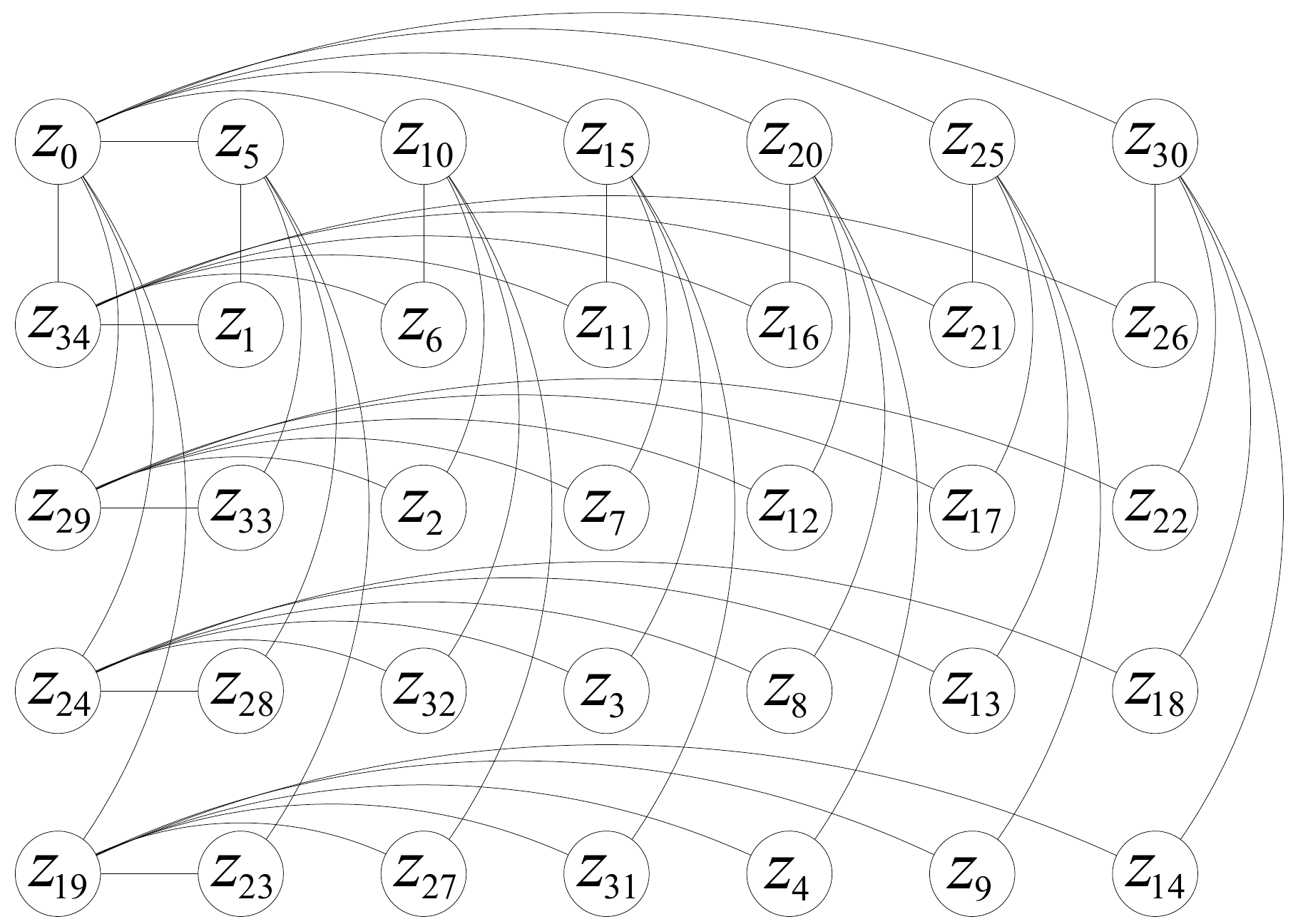}\hspace{0.5cm}
    \includegraphics[width=3in]{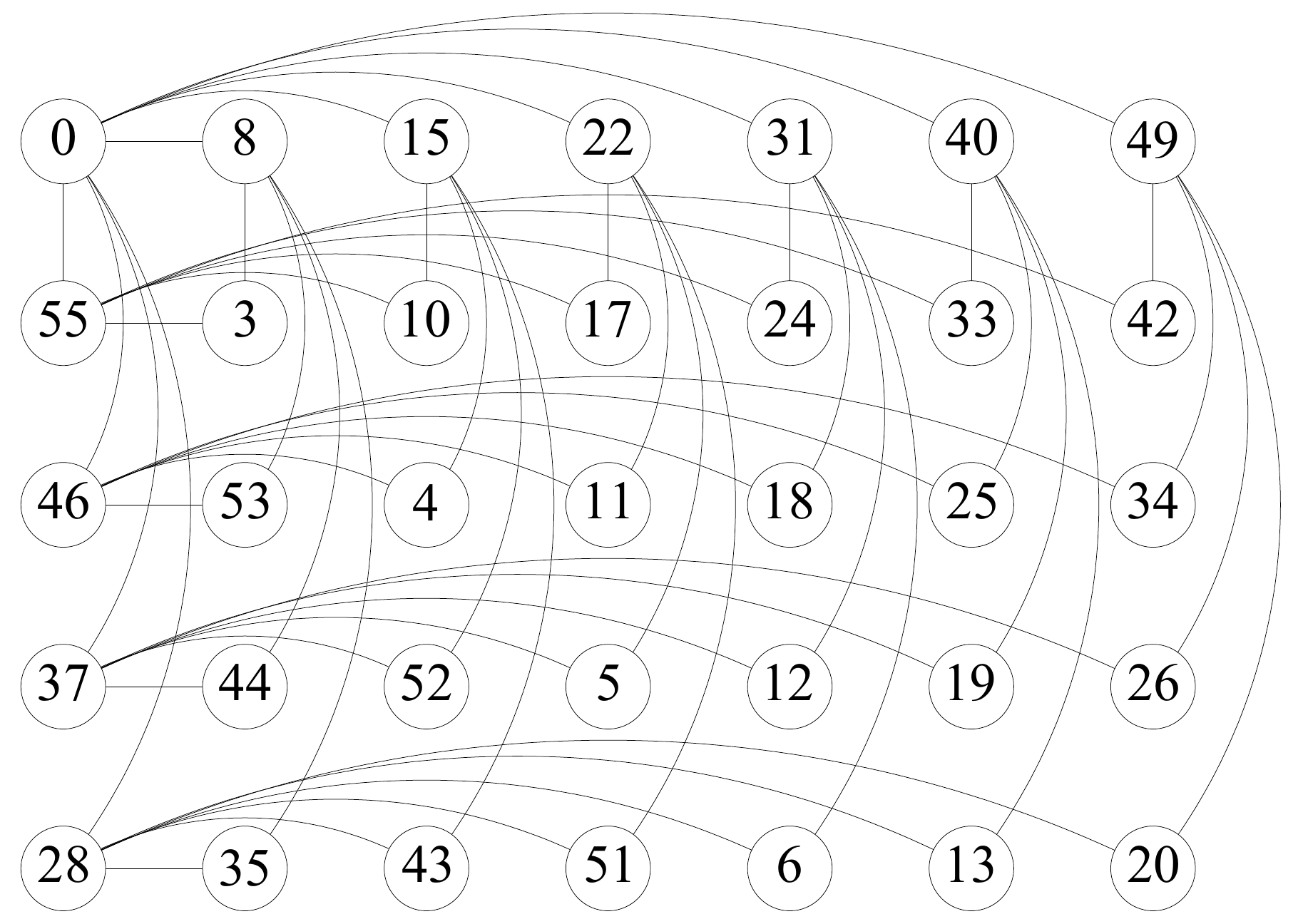}
    \caption{An ordering (left) and an optimal radio labeling (right) of $K_{1,6} \Box K_{1,4}$} 
    \label{Fig1}
\end{figure}

\begin{theorem}\label{thm:ps} Let $m,n \geq 3$ be integers. Then
\begin{equation}
\label{rn:ps}
\rn(P_m \Box K_{1,n}) = \left\{
\begin{array}{ll}
\frac{1}{2}\left[m^2(n+1)+2m+n-1\right], & \mbox{ if $m$ is odd}, \\ [0.2cm]
\frac{1}{2}\left[m^2(n+1)+2(m-1)\right], & \mbox{ if $m$ is even}.
\end{array}
\right.    
\end{equation}
\end{theorem}
\begin{proof} Denote $G = P_m \Box K_{1,n}$.  Then $p = |V(G)|= m(n+1)$, diam$(G)=m+1$,  $|W(G)|=1$ if $m$ is odd, and $|W(G)|=2$ if $m$ is even. As $L(K_{1,n}) = n$, 
$
L (P_m) = 
\frac{m^2-1}{4}$ if $m$ is odd,  and $L (P_m) =\frac{m(m-2)}{4}$ if $m$ is even, 
we obtain the right-hand side of \eqref{rn:ps} as a lower bound for $\rn(G)$ by substituting these into \eqref{eqn:lower}. 

Next we give a feasible  ordering $\vec{V}=(z_0, z_1, \cdots, z_{p-1})$ for $V(G)$, satisfying Theorem \ref{thm:upper1} when $m$ is odd, and satisfying Theorem \ref{thm:suff} when $m$ is even. Denote $V(P_m) = \{x_1,x_2,\ldots,x_m\}$, $E(P_m) = \{x_ix_{i+1} : i \in [1, m-1]\}$, $V(K_{1,n}) = \{y_i: i \in [0,n]\}$, and  $E(K_{1,n}) = \{y_0y_j : j \in [1, n]\}$.

\noindent
\textsf{Case-1:} $m$ is odd. Denote $m' = (m+1)/2$.  Then $W(G) = \{(x_{m'}, y_0)\}$.  Define $\vec{V}$ 
by two steps: 

\noindent
{\bf Step 1.} For $t \in [0, 3n+2]$, define $z_t:=(x_i, y_j)$, $i\in\{1, m', m\}$ and $j \in[0,n]$:

\textsf{Subcase-1:} $n \equiv 0$  (mod $3$).  
\begin{equation*}
t := \left\{
\begin{array}{lll}
3n+2, &\mbox{$i=1$ and $j = 0$}, \\ [0.2cm]
n+j-1, &\mbox{$i=1$ and $j > 0$ and $j \equiv 0$ (mod 3); or $i=m$ and $j \equiv 2$ (mod 3)}, \\ 
2n+j+1, &\mbox{$i=1$ and $j \equiv 1$ (mod 3); or $i=m'$ and $j \equiv 2$ (mod 3)}, \\ 
j, &\mbox{$i=1$ and $j \equiv 2$ (mod 3); or $i=m'$ and $j \equiv 0$ (mod 3);} \\
& \hspace{0.2in}\mbox{or $i=m$ and $j \equiv 1$ (mod 3)}, \\ n+j+2, &\mbox{$i=m'$ and $j \equiv 1$ (mod 3)}, \\ 
3n+1, &\mbox{$i=m$ and $j = 0$}, \\ 
2n+j-2, &\mbox{$i=m$ and $j > 0$ and $j \equiv 0$ (mod 3)}. \\ 
\end{array}
\right.
\end{equation*}

\textsf{Subcase-2:} $n \equiv 1$ (mod $3$). 
\begin{equation*}
t := \left\{
\begin{array}{ll}
3n+2, &\mbox{$i=1$ and $j = 0$}, \\ 
6\lfloor n/3 \rfloor+j+2, &\mbox{$i=1$ and $j > 0$ and $j \equiv 0$ (mod 3); or $i=m'$ and $j \equiv 1$ (mod 3);} \\
&\hspace{0.2in}\mbox{or $i=m$ and $j \equiv 2$ (mod 3)}, \\ 
3\lfloor n/3 \rfloor+j+1, &\mbox{$i=1$ and $j \equiv 1$ (mod 3); or $i=m'$ and $j \equiv 2$ (mod 3);}\\
&\hspace{0.2in} \mbox{or $i=m$ and $j>0$ and $j \equiv 0$ (mod 3)}, \\ 
j, &\mbox{$i=1$ and $j \equiv 2$ (mod 3); or $i=m'$ and $j \equiv 0$ (mod 3);} \\ 
 & \hspace{0.2in} \mbox{or $i=m$ and $j \equiv 1$ (mod 3)}, \\ 
3n+1, &\mbox{$i=m$ and $j = 0$}. \\ 
\end{array}
\right.
\end{equation*}

\textsf{Subcase-3:} $n \equiv 2$ (mod $3$).
\begin{equation*}
t := \left\{
\begin{array}{ll}
3n+2, &\mbox{$i=1$ and $j = 0$}, \\ [0.2cm]
3\lfloor n/3 \rfloor+j+2, &\mbox{$i=1$ and $j > 0$ and $j \equiv 0$ (mod 3); or $i=m'$ and $j \equiv 1$ (mod 3);} \\
& \mbox{\hspace{0.5cm}or $i=m$ and $j \equiv 2$ (mod 3)}, \\ [0.2cm]
6\lfloor n/3 \rfloor+j+4, &\mbox{$i=1$ and $j \equiv 1$ (mod 3); or $i=m'$ and $j \equiv 2$ (mod 3)}; \\
& \mbox{\hspace{0.5cm}or $i=m$ and $j > 0$ and $j \equiv 0$ (mod 3)}, \\ [0.2cm]
j, &\mbox{$i=1$ and $j \equiv 2$ (mod 3); or $i=m'$ and $j \equiv 0$ (mod 3);} \\ 
& \mbox{\hspace{0.5cm}or $i=m$ and $j \equiv 1$ (mod 3)}, \\[0.2cm]
3n+1, &\mbox{$i=m$ and $j = 0$}. \\ [0.2cm]
\end{array}
\right.
\end{equation*}

\noindent
{\bf Step 2.} For $t \in [3(n+1), p-1]$, set $z_t := (x_i,y_j)$, $i \in [2, m-1]$, $i \neq m'$, $j \in [0, n]$, where 
\begin{equation*}
t := \left\{
\begin{array}{ll}
(2i+1)(n+1)-1, &\mbox{$i \in [2, m'-1]$,   $j=0$}, \\ [0.2cm]
3(n+1)+2\lceil n/2 \rceil+2(i-2)(n+1)+j-1, &\mbox{$i \in [2, m'-1]$, $j$ is even and $j \neq 0$}, \\ [0.2cm]
3(n+1)+2(i-2)(n+1)+j, &\mbox{$i \in [2, m'-1]$, $j$ is odd}, \\ [0.2cm]
(2i-m)(n+1), &\mbox{$i \in [m'+1, m-1]$, $j = 0$}, \\ [0.2cm]
3(n+1)+(2i-m-3)(n+1)+j, &\mbox{$i \in [m'+1, m-1]$, $j$ is even and $j \neq 0$}, \\ [0.2cm]
3(n+1)+2 \lfloor n/2 \rfloor+(2i-m-3)(n+1)+j+1, &\mbox{$i \in [m'+1, m-1]$, $j$ is odd}.
\end{array}
\right.
\end{equation*}
Observe that the above defined ordering $\vec{V}$ is feasible and $L(z_0)+L(z_{p-1})=1$. It suffices to show that $\vec{V}$ satisfies \cref{eqn:dab} for Theorem \ref{thm:upper1} (b).  Let $z_a$ and $z_b$ be two arbitrary vertices, $0 \leq a < b \leq p-1$. Denote the right-hand side of \eqref{eqn:dab} by $S_{a,b}$. Assume  $b-a \geq 3$. For any $t$, among the three values in $\{d(z_t,z_{t+1}) : t \in [a,a+2]\}$, two are at most $(d+2)/2$, and the remaining one is at most $d$. Thus,
$$
\begin{array}{llll}
S_{a,b} &\leq&  \left[2(b-a)/3\right]\left[(d+2)/2\right]+\left[(b-a)/3\right]d-(b-a)(d+1)+d+1 \\
&=& \left[(b-a)/3\right](-d-1)+d+1 \leq 0 < d(z_a,z_b). 
\end{array}
$$
Assume $b-a = 2$. If $a \in [0, 3n-2]$. Suppose $a \equiv 0$ (mod $3$). By the defined ordering we  have $d/2 \leq d(z_a,z_{a+1}) = L(z_a)+L(z_{a+1}) \leq (d+2)/2, d(z_{a+1},z_{b}) = L(z_{a+1})+L(z_b) = d$  and $d(z_a,z_b) = (d+2)/2$. Hence, $S_{a,b} \leq d/2 < d(z_a,z_b)$. If $a \equiv 1$ (mod $3$), then $d(z_a,z_{a+1}) = L(z_a)+L(z_{a+1}) = d$, $d(z_{a+1},z_b) = L(z_{a+1})+L(z_b) = (d+2)/2$ 
and $d(z_a,z_b) = (d+2)/2$. Hence  $S_{a,b} = d/2 < d(z_a,z_b)$. If $a \equiv 2$ (mod $3$), then $d(z_t,z_{t+1}) = L(z_t)+L(z_{t+1}) = (d+2)/2$ for $t=a,a+1$,  and $d(z_a,z_b) = d$. Hence,  $S_{a,b} = 1 < d(z_a,z_b)$. 

If $a = 3n-1$, then $d(z_a,z_{a+1}) = L(z_a)+L(z_{a+1}) = (d+2)/2$, $d(z_{a+1},z_b) = L(z_{a+1})+L(z_b) = d/2$ and $d(z_a,z_b) = d-1$. Hence $S_{a,b} = 0 < d(z_a,z_b)$. If $a = 3n$, then  $d(z_a,z_{a+1}) = L(z_a)+L(z_{a+1}) = d/2$, $d(z_{a+1},z_b) = L(z_{a+1})+L(z_b) = d-2$ and $d(z_a,z_b) = d/2$. Hence, $S_{a,b} = (d/2)-3 < d(z_a,z_b)$. If $a = 3n+1$, then $d(z_a,z_{a+1}) = d-2$,  $d(z_{a+1},z_b) = d/2$ and $d(z_a,z_b) = d/2-2$. Again, $S_{a,b} \leq (d/2)-3 < d(z_a,z_b)$. 

If $3n+2 \leq a \leq p-1$, then $d(z_t,z_{t+1}) \leq (d+2)/2$ for $t=a,a+1$ and $d(z_a,z_b) \geq 1$. Hence, $S_{a,b} \leq 1 \leq d(z_a,z_b)$. Thus, \cref{eqn:dab} of  Theorem \ref{thm:upper1} (b) is satisfied.

\textsf{Case-2:} $m$ is even. Denote $m=2m'$.  Recall that $W(P_m) = \{x_{m'},x_{m'+1}\}$, and $|W(P_m \Box K_{1,n})|=2$. Then $\xi=1$. Define an ordering $\vec{V} = (z_0,z_1,\ldots, z_{p-1})$  by: 

\begin{equation*}
t := \left\{
\begin{array}{ll}
2(m'-i)(n+1)+2\lfloor n/2 \rfloor+j+1, &\mbox{$i \in[1, m']$ and $j$ is odd}, \\ [0.2cm]
2(m'-i)(n+1)+j, &\mbox{$i \in[1, m']$, $j$ is even}, \\ [0.2cm]
2(m-i)(n+1)+j, &\mbox{$i \in [m'+1, m]$ and $j$ is odd}, \\ [0.2cm]
2(m-i)(n+1)+2\lceil n/2 \rceil+j-1, &\mbox{$i \in [m'+1, m]$, $j$ is even and $j \neq 0$}, \\ [0.2cm]
2(m-i+1)(n+1)-1, &\mbox{$i \in [m'+1, m]$ and $j = 0$}.
\end{array}
\right.
\end{equation*}

\noindent
Observe that the ordering $\vec{V}$ defined above is feasible and satisfies  conditions (a) and (b)-(ii) of Theorem \ref{thm:suff}, that is, $d(z_t, z_{t+1}) \leq (d+\xi+2)/2 = (m+4)/3$ holds for all $t$.  
\qed
\end{proof}

Figures \ref{P5K13}-\ref{P6K15} provide examples for labellings in  \cref{thm:ps}. 

\begin{figure}[ht!]
\centering
\includegraphics[width=2.3in]{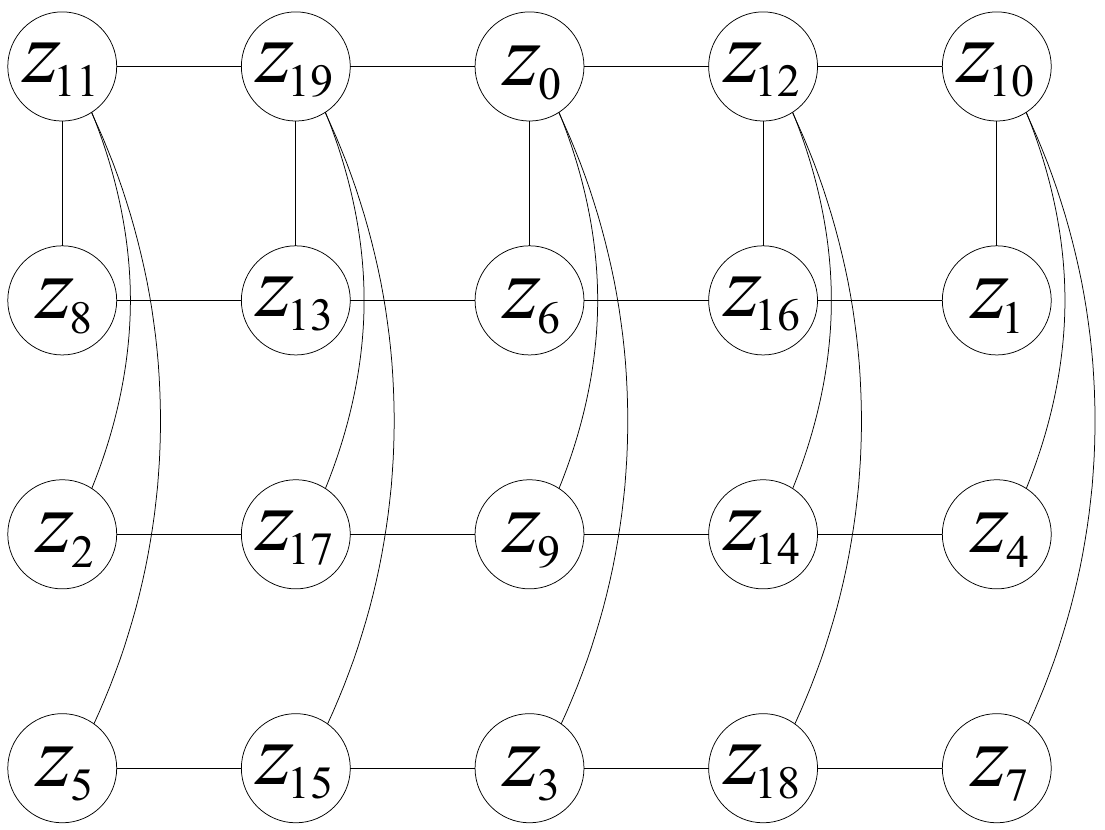}\hspace{1cm}
\includegraphics[width=2.3in]{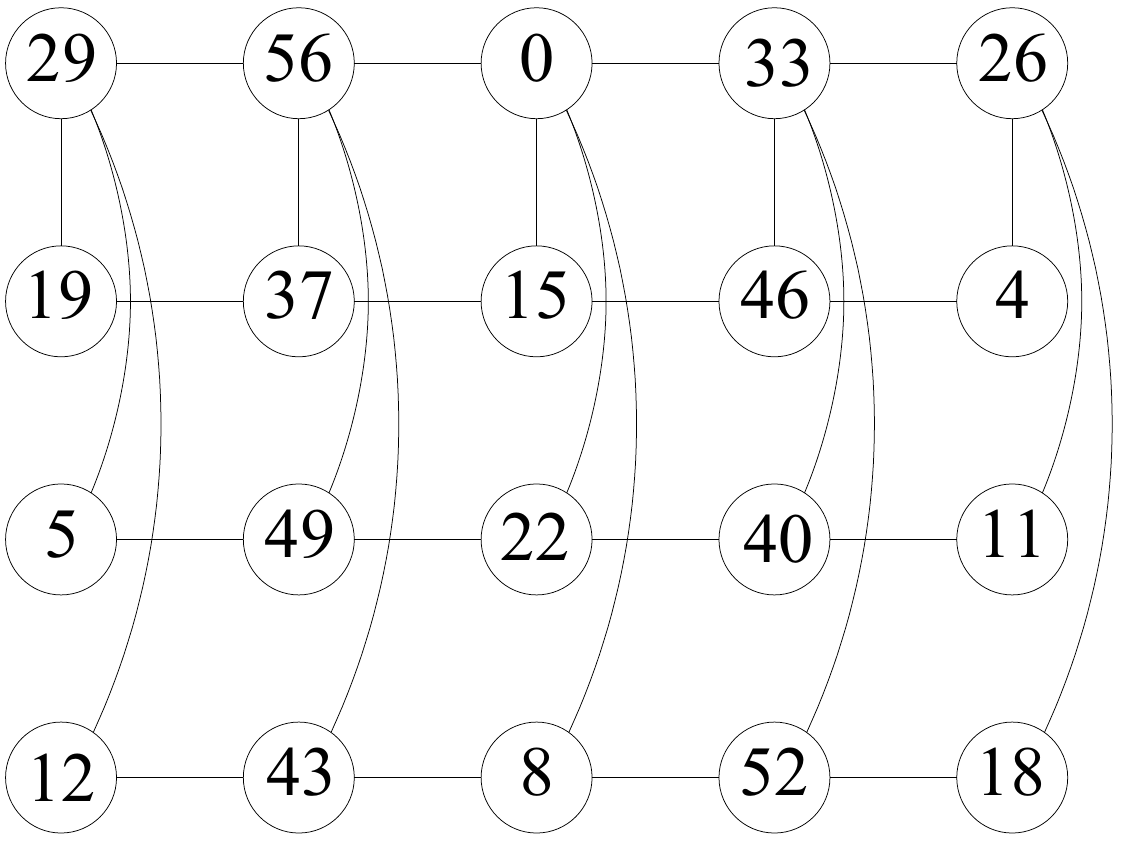}
\caption{An optimal ordering and an optimal radio labeling of $P_5 \Box K_{1,3}$}
\label{P5K13}
\end{figure}

\begin{figure}[ht!]
\centering
\includegraphics[width=2.4in]{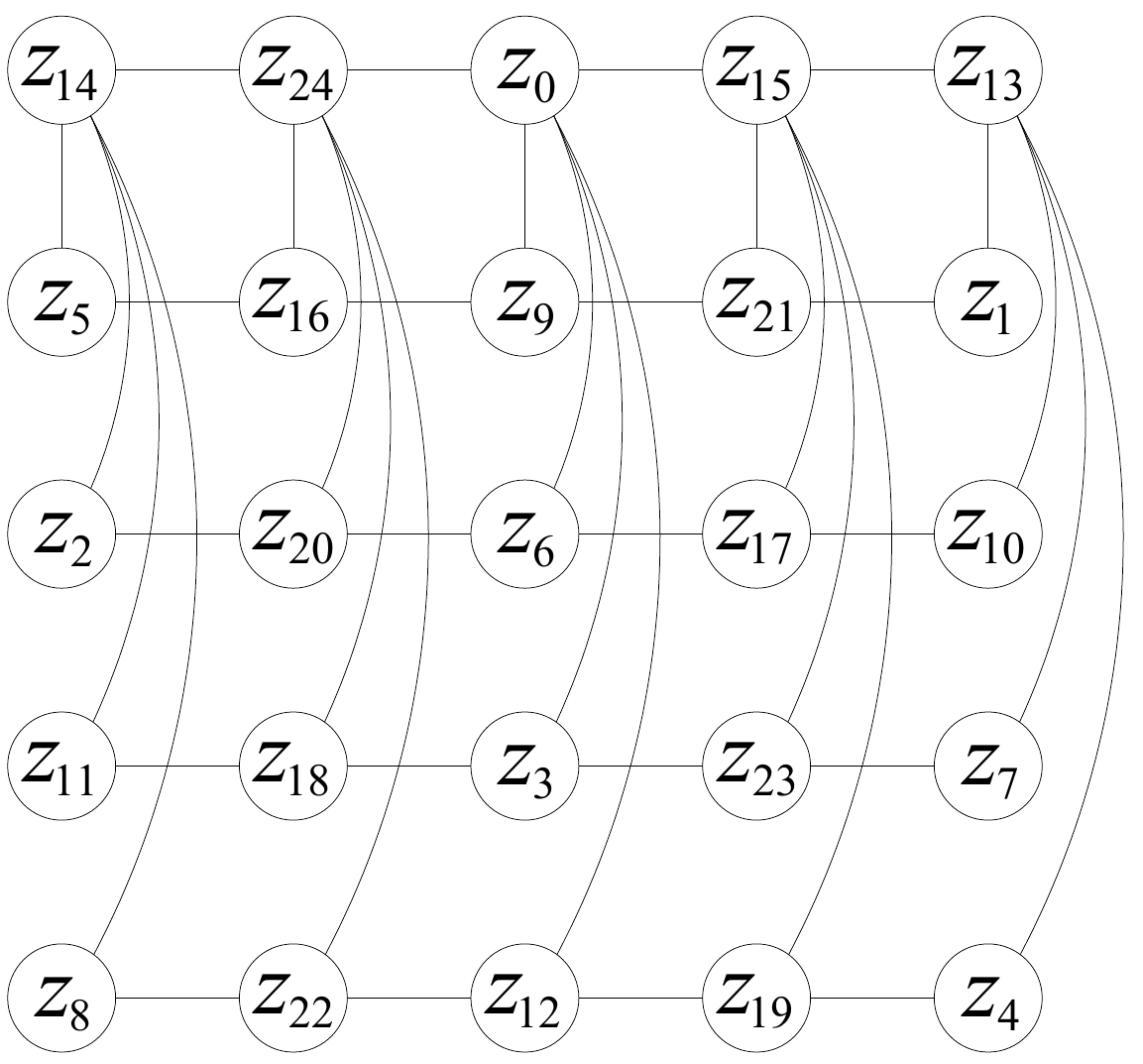}\hspace{1cm}
\includegraphics[width=2.4in]{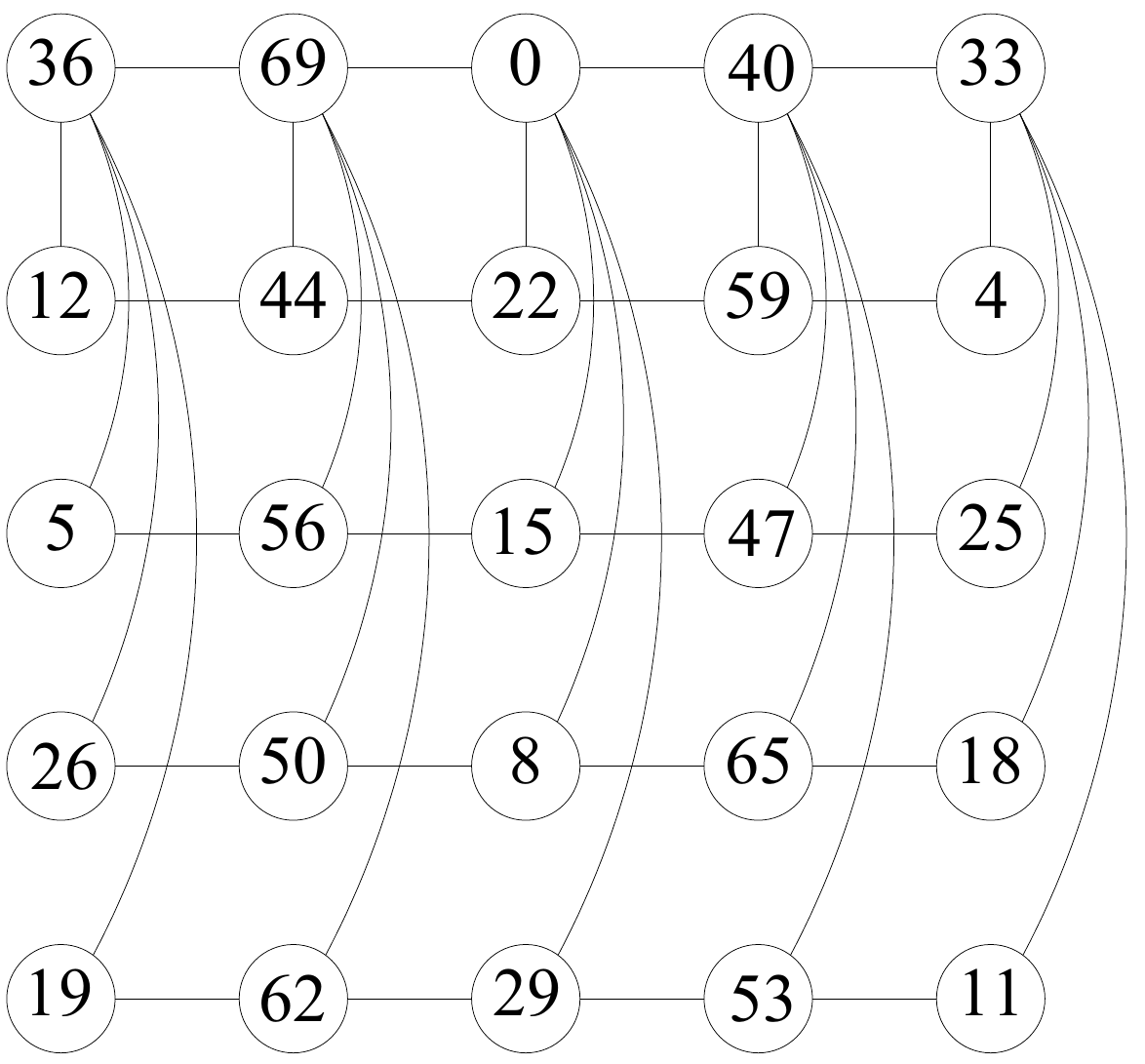}
\caption{An optimal ordering and an optimal radio labeling of $P_5 \Box K_{1,4}$}
\label{P5K14}
\end{figure}

\begin{figure}[ht!]
\centering
\includegraphics[width=2.5in]{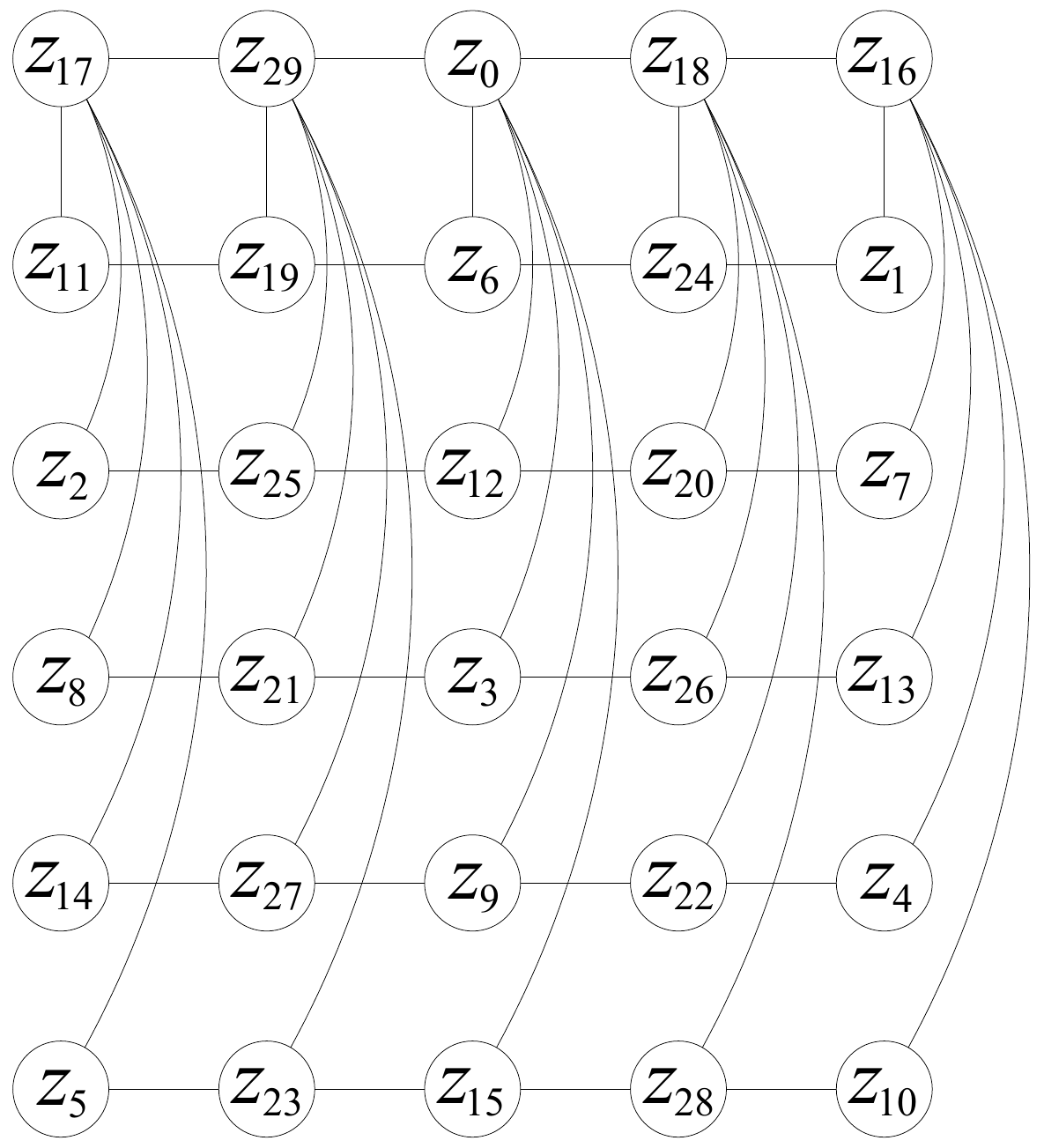}\hspace{0.5cm}
\includegraphics[width=2.5in]{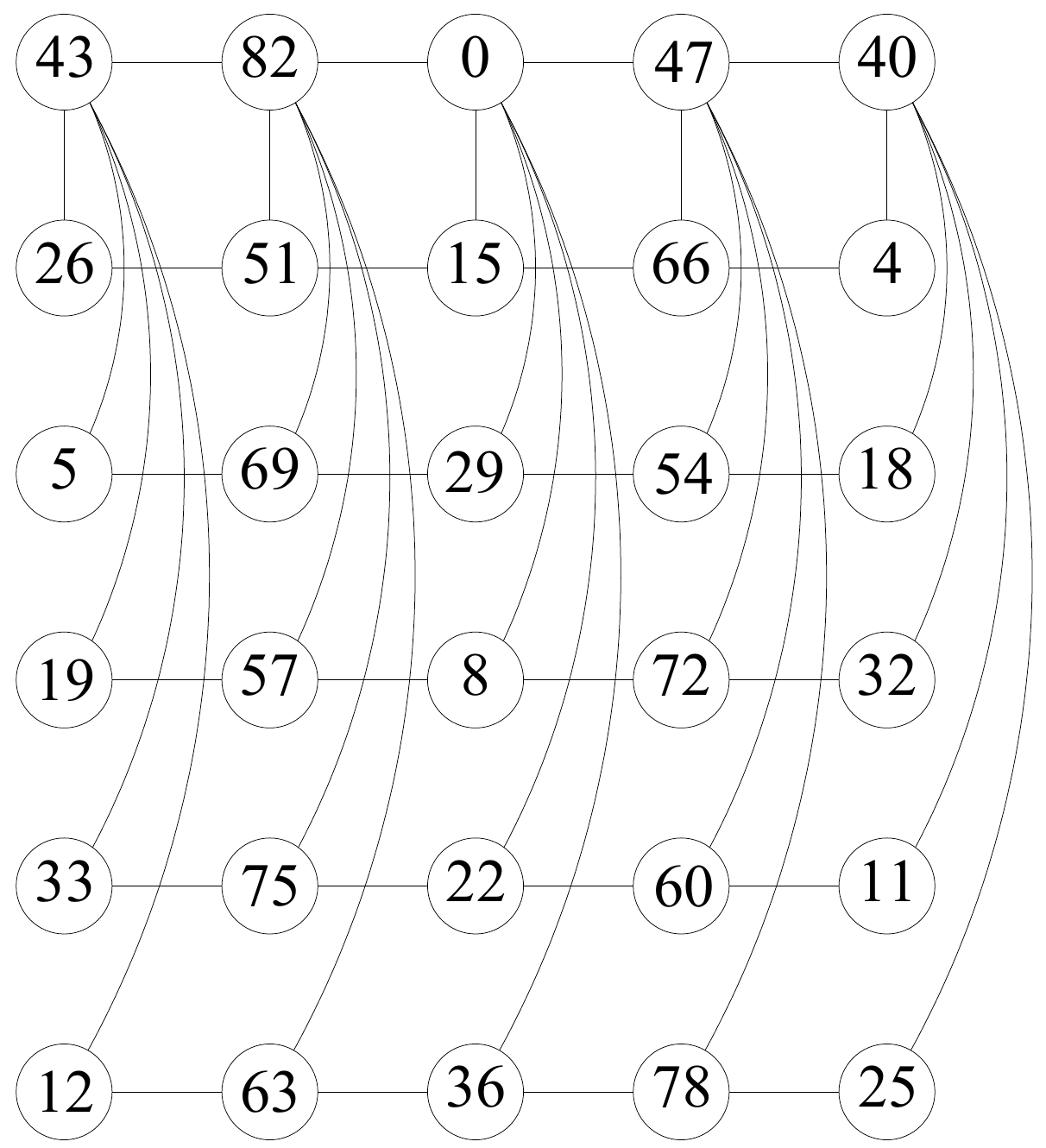}
\caption{An optimal ordering and an optimal radio labeling of $P_5 \Box K_{1,5}$}
\label{P5K15}
\end{figure}

\begin{figure}[ht!]
\centering
\includegraphics[width=3in]{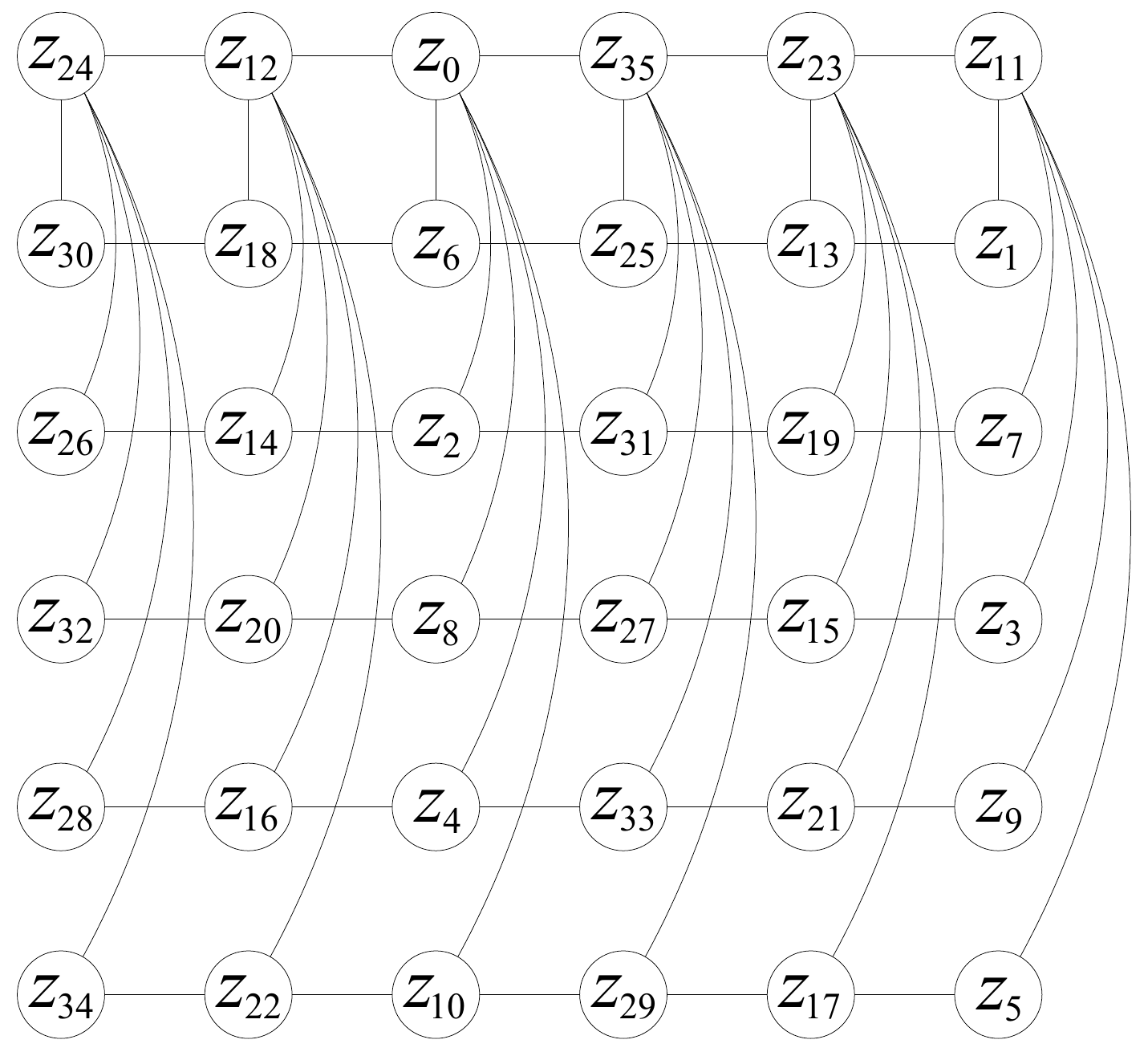}\hspace{0.5cm}
\includegraphics[width=3in]{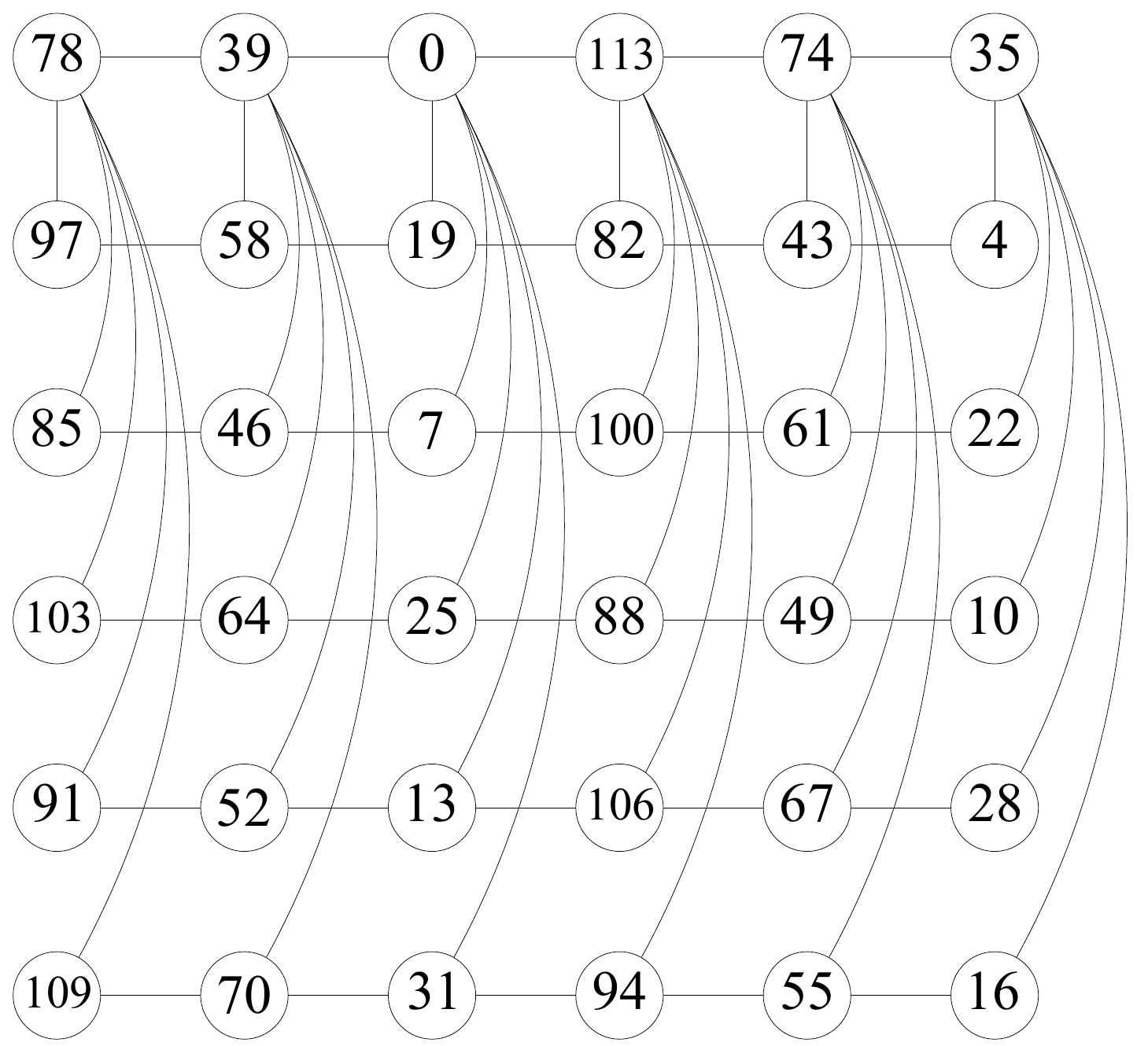}
\caption{An optimal ordering and an optimal radio labeling of $P_6 \Box K_{1,5}$}
\label{P6K15}
\end{figure}


\end{document}